\documentclass[reqno,oneside,11pt,english]{amsart}
\usepackage{amsmath,amsfonts,amssymb,amsthm,epsfig}
\newcommand{\RNum}[1]{\uppercase\expandafter{\romannumeral #1\relax}}
\usepackage{comment}

\usepackage{color}
\usepackage[pagebackref,colorlinks,citecolor=blue,linkcolor=blue]{hyperref}
\usepackage{graphicx}
\usepackage{titletoc,epsf}
\usepackage{amsmath,amsfonts,latexsym,amsthm,amsxtra,amssymb,bbm}
\allowdisplaybreaks
\usepackage{graphicx,float}
\usepackage{tikz}
\usetikzlibrary{intersections,patterns}
\usetikzlibrary{calc,spy}
\usepackage{ifthen}
\usepackage{float}
\usepackage{cite}
\allowdisplaybreaks

\newtheorem{thm}{Theorem}[section]
\newtheorem{lem}[thm]{Lemma}
\newtheorem{op}[thm]{Open Problem}
\newtheorem{cor}[thm]{Corollary}
\newtheorem{prop}[thm]{Proposition}
\newtheorem{step}{Step}[section]

\providecommand{\vint}[1]{\mathchoice
	{\mathop{\vrule width 5pt height 3 pt depth -2.5pt
			\kern -9pt \kern 1pt\intop}\nolimits_{\kern -5pt{#1}}}
	{\mathop{\vrule width 5pt height 3 pt depth -2.6pt
			\kern -6pt \intop}\nolimits_{\kern -3pt{#1}}}
	{\mathop{\vrule width 5pt height 3 pt depth -2.6pt
			\kern -6pt \intop}\nolimits_{\kern -3pt{#1}}}
	{\mathop{\vrule width 5pt height 3 pt depth -2.6pt
			\kern -6pt \intop}\nolimits_{\kern -3pt{#1}}}}

\newtheorem{cl}{Claim}[section]

\newtheorem{ca}{Case}[section]
\newtheorem{sca}[ca]{Sbucase}
\newtheorem{scl}[cl]{Subclaim}
\newtheorem{conj}{Conjecture}

\theoremstyle{definition}
\newtheorem{defn}[thm]{Definition}
\newtheorem*{ques*}{Question}
\newtheorem{ques}{Question}
\newtheorem{rem}[thm]{Remark}
\newtheorem{exam}[thm]{Example}

\newcounter {own}
\def\theown {\thesection       .\arabic{own}}
\numberwithin{equation}{section}

\newenvironment{pf}[1][]{%
	\vskip 3mm
	\noindent
	\ifthenelse{\equal{#1}{}}%
	{{\slshape Proof. }}%
	{{\slshape #1.} }%
}%
{\qed\bigskip}


\newcommand{\HH}{{\mathcal H}}
\newcommand{\IR}{{\mathbb R}}
\newcommand{\IN}{{\mathbb N}}

\newcommand{\Lip}{{\operatorname{Lip}\,}}

\newcommand{\diam}{{\operatorname{diam}\,}}
\newcommand{\spt}{{\operatorname{spt}}}

\newcommand{\dist}{{\operatorname{dist}}}

\newcommand{\bas}{\begin{assertion}}
	\newcommand{\eas}{\end{assertion}}
\newcommand{\ben}{\begin{enumerate}}
	\newcommand{\een}{\end{enumerate}}
\newcommand{\bst}{\begin{step}}
	\newcommand{\est}{\end{step}}

\def\be{\begin{equation}}
	\def\ee{\end{equation}}

\newcommand{\bee}{\begin{enumerate}}
	\newcommand{\eee}{\end{enumerate}}

\newcommand{\blem}{\begin{lem}}
	\newcommand{\elem}{\end{lem}}
\newcommand{\bthm}{\begin{thm}}
	\newcommand{\ethm}{\end{thm}}
\newcommand{\bcor}{\begin{cor}}
	\newcommand{\ecor}{\end{cor}}
\newcommand{\beg}{\begin{exam}}
	\newcommand{\eeg}{\end{exam}}
\newcommand{\begs}{\begin{examples}}
	\newcommand{\eegs}{\end{examples}}
\newcommand{\bdefe}{\begin{defn}}
	\newcommand{\edefe}{\end{defn}}
\newcommand{\bprob}{\begin{prob}}
	\newcommand{\eprob}{\end{prob}}
\newcommand{\bques}{\begin{ques}}
	\newcommand{\eques}{\end{ques}}
\newcommand{\bei}{\begin{itemize}}
	\newcommand{\eei}{\end{itemize}}
\newcommand{\bcon}{\begin{conj}}
	\newcommand{\econ}{\end{conj}}
\newcommand{\bop}{\begin{op}}
	\newcommand{\eop}{\end{op}}

\newcommand{\bstep}{\begin{step}}
	\newcommand{\estep}{\end{step}}

\newcommand{\bca}{\begin{ca}}
	\newcommand{\eca}{\end{ca}}
\newcommand{\bsca}{\begin{sca}}
	\newcommand{\esca}{\end{sca}}

\newcommand{\bcl}{\begin{cl}}
	\newcommand{\ecl}{\end{cl}}

\newcommand{\bscl}{\begin{scl}}
	\newcommand{\escl}{\end{scl}}

\newcommand{\bcons}{\begin{conjs}}
	\newcommand{\econs}{\end{conjs}}
\newcommand{\bprop}{\begin{prop}}
	\newcommand{\eprop}{\end{prop}}
\newcommand{\br}{\begin{rem}}
	\newcommand{\er}{\end{rem}}
\newcommand{\brs}{\begin{rems}}
	\newcommand{\ers}{\end{rems}}
\newcommand{\bo}{\begin{obser}}
	\newcommand{\eo}{\end{obser}}
\newcommand{\bos}{\begin{obsers}}
	\newcommand{\eos}{\end{obsers}}
\newcommand{\bpf}{\begin{pf}}
	\newcommand{\epf}{\end{pf}}
\newcommand{\ba}{\begin{array}}
	\newcommand{\ea}{\end{array}}
\newcommand{\beq}{\begin{eqnarray}}
	\newcommand{\beqq}{\begin{eqnarray*}}
		\newcommand{\eeq}{\end{eqnarray}}
	\newcommand{\eeqq}{\end{eqnarray*}}

\newcounter{minutes}
\divide\time by 60
\newcounter{hours}
\multiply\time by 60 \addtocounter{minutes}{-\time}

\usepackage{fancyhdr}
\newcommand{\papername}{\scriptsize {MANZI HUANG, PANU LAHTI, JIANG LI, AND ZHUANG WANG}}    
\newcommand{\papertitle}{\scriptsize {BOXING INEQUALITY AND POINCAR\'E-TYPE INEQUALITIES ON JOHN DOMAINS WITH BBM FACTOR}} 

\fancypagestyle{plain}{%
  \fancyhf{}
  \fancyhead[C]{\ifthenelse{\isodd{\value{page}}}{\papername}{\papertitle}}
  \fancyfoot[C]{\thepage}
  
}

\begin{document}

	\title[]{Boxing inequalities for relative fractional perimeter and fractional Poincar\'e-type  inequalities on John domains with the BBM factor}


	\author{Manzi Huang, Panu Lahti, Jiang Li$^{*}$, and Zhuang Wang}

		\address[Manzi Huang]{Key Laboratory of Computing and Stochastic Mathematics (Ministry of Education), School of Mathematics and Statistics, Hunan Normal University, Changsha, Hunan 410081, P.R. China.}
	\email{mzhuang@hunnu.edu.cn}
	
    \address[Panu Lahti]{Academy of Mathematics and Systems Science, 
        Chinese Academy of Sciences, Beijing 100190, P.R. China.}
	\email{panulahti@amss.ac.cn}
	
	\address[Jiang Li]{Key Laboratory of Computing and Stochastic Mathematics (Ministry of Education), School of Mathematics and Statistics, Hunan Normal University, Changsha, Hunan 410081, P.R. China.}
	\email{jiang\_li@hunnu.edu.cn; jiangli.math@qq.com}

	\address[Zhuang Wang]{Key Laboratory of Computing and Stochastic Mathematics (Ministry of Education), School of Mathematics and Statistics, Hunan Normal University, Changsha, Hunan 410081, P.R. China.}
	\email{zwang@hunnu.edu.cn}

	\date{}

	\thanks{$^*$Corresponding author: Jiang Li}

\begin{abstract}
	For \(0<\delta,\tau<1\) and \(1\le s\le \frac{n}{n-\delta}\), we prove that for a given \(s\)-John domain
	\(\Omega\subset \IR^n\), the following Boxing inequality holds for every Lebesgue measurable set \(U\subset\Omega\) with \(|U|/|\Omega|\le\gamma<1\):
	\[
	\mathcal{H}^{s(n-\delta)}_{\infty}(U\setminus\mathcal{N}_U)\le C(1-\delta)\int_\Omega\int_{|x-y|<\tau\operatorname{dist}(y,\partial\Omega)}\frac{|\chi_U(x)-\chi_U(y)|}{|x-y|^{n+\delta}}\,dx\,dy,
	\]
	where \(\mathcal{H}^{s(n-\delta)}_{\infty}(U)\) denotes the \(s(n-\delta)\)-dimensional Hausdorff content of \(U\), \(\mathcal{N}_U\) is a set of Lebesgue measure zero and the constant \(C\) depends only on \(n,\tau,s,\gamma\), the John constant and the diameter of \(\Omega\). Moreover, we establish the functional formulation of the above Boxing inequality and discuss the equivalence between these two formulations. Based on the Boxing inequality, we prove the fractional Poincaré--Wirtinger trace inequality on \(s\)-John domains, of which the fractional Sobolev--Poincaré inequality and fractional Hardy-type inequality are special cases. Notably, we prove all of the aforementioned inequalities with the Bourgain--Brezis--Mironescu (BBM) factor \(1-\delta\). Furthermore, with the aid of the Bourgain--Brezis--Mironescu formula, we recover the Poincaré--Wirtinger trace inequality. Finally, by showing that, under the separation property, any domain supporting the Boxing inequality is necessarily a John domain, we conclude that the John domain condition is essentially sharp for the above inequalities. All the above inequalities with the BBM factor are new even for Lipschitz domains.
\end{abstract}

\date{\today}


\maketitle

\tableofcontents

\section{\large Introduction}
Let $\IR^n$ be the $n$-dimensional Euclidean space with $n\geq 2$. The classical Boxing inequality,  first proved by Gustin \cite[page 232]{Gustin60}, can be stated as follows. There exists a constant $C=C(n)>0$ such that  for every bounded open set $U\subset \IR^n$  with smooth boundary, one can find a covering $\{B(x_i,r_i)\}_{i=1}^{\infty}$ of $U$ for which
\begin{equation}\label{eq:inte-boxing-whole}
    \sum_{i=1}^{\infty}r_i^{n-1}\leq C~\text{Per}(U).
\end{equation}
Here $\text{Per}(U)$ denotes the perimeter of $U$. By definition (cf.\cite[page 194]{EGbook}), this perimeter coincides with the 
$(n-1)$-dimensional Hausdorff measure of the topological boundary of 
$U$, denoted by $\HH^{n-1}(\partial U)$
{(for definitions, see Section \ref{sec:prelis}).}

For any Lebesgue measurable set $U\subset \IR^n$, define
\[
P_\delta(U):=2\int_{U}\int_{\IR^n\setminus U}\frac{1}{|x-y|^{n+\delta}}\,dx\,dy=[\chi_U]_{W^{\delta, 1}(\IR^n)},
\]
where $[\chi_U]_{W^{\delta, 1}(\IR^n)}$ is the fractional Sobolev energy of $\chi_{U}$; see Section \ref{FP}.  The quantity $P_\delta$  appeared in \cite{Vis91,Mazya03} and has been called the {\it fractional perimeter} \cite{CaRoSa10} or {\it non-local $\delta$-perimeter} \cite{BLP14}. 

Ponce--Spector  \cite[Theorem 1.2]{PS20} studied the Boxing inequality related to the fractional perimeter $P_\delta$. They showed that there exists a constant $C=C(n)>0$ such that  for every bounded open set $U\subset \IR^n$, one can find a covering $\{B(x_i,r_i)\}_{i=1}^{\infty}$ of $U$ for which
\begin{equation}\label{frac-boxing-whole}
	\sum_{i=1}^{\infty}r_i^{n-\delta}\leq C\delta(1-\delta)P_\delta(U).
\end{equation}
According to the commendable work by Davil\'a \cite[Theorem 1]{Davila02} and Maz'ya--Shaposhnikova \cite[Theorem 3]{MS02}, $P _\delta(U)$ satisfies the  asymptotics
\begin{equation}\label{BBM-MT-Theorem}
	\lim_{\delta\rightarrow 1^-}(1-\delta)P _\delta(U)=C^{\prime}(n)\,\text{Per}(U) \qquad\text{and}\qquad \lim_{\delta\rightarrow 0^+}\delta P _\delta(U)=C^{\prime\prime}(n)\,|U|,
\end{equation}
where $|U|$ denotes the Lebesgue measure of $U$.
Therefore, the Boxing inequality \eqref{frac-boxing-whole} allows one to recover the endpoint cases (cf. \cite[page 123]{PS20}): one endpoint (letting $\delta\rightarrow 1^-$) is the classical Boxing inequality \eqref{eq:inte-boxing-whole}  and the other (letting $\delta\rightarrow 0^+$) yields $\sum_{i=1}^{\infty}r_i^{n}\leq C|U|$, which follows  directly from the definition of Lebesgue measure. Moreover, these asymptotic results  indicate that the factor $\delta(1-\delta)$ on the right-hand side of \eqref{frac-boxing-whole} is essentially sharp.

Replacing $\mathbb R^n$ by a domain $\Omega\subset \mathbb R^n$ in the definition of $P_\delta$, one obtains the relative fractional perimeter $P_\delta(U, \Omega)$ of a Lebesgue measurable set $U\subset \Omega$; see Section \ref{FP} for details.
The Boxing inequality related to the   relative fractional  perimeter was established by Ponce--Spector \cite[Proposition 4.5]{PS20}, reading as follows. Let $\Omega\subset \IR^n$ be a smooth domain  and $\delta,\gamma\in (0,1).$ Then there exists a constant $C=C(\gamma, n, \Omega)$  such that for every open set $U\subset \Omega$ with $|U|/|\Omega|\leq \gamma,$ one can find a covering $\{B(x_i,r_i)\}_{i=1}^{\infty}$ of $U$ for which
\begin{equation}\label{frac-boxing-local}
	\sum_{i=1}^{\infty}r_i^{n-\delta}\leq C(1-\delta)P_\delta(U,\Omega).
\end{equation}

Unlike \eqref{frac-boxing-whole}, this inequality \eqref{frac-boxing-local} contains no factor $\delta$ in the right-hand side, which is not surprising and can be seen from Maz'ya--Shaposhnikova \cite[Corollary 1]{MS02}.

In the Boxing inequality \eqref{frac-boxing-local}, $\Omega$ is assumed to be a smooth domain. Therefore, a very natural question is 
\begin{ques*}
	For which more general domains does a Boxing inequality of the form \eqref{frac-boxing-local} hold? 
\end{ques*}

It is expected that such an inequality holds for Lipschitz domains, which will be shown to be true in this paper. In fact, we consider John domains, which include smooth domains, Lipschitz domains and even certain cusp domains. We establish a Boxing inequality of the form \eqref{frac-boxing-local} on John domains and, in turn, show that John domains are essentially the
only
domains supporting this inequality. Moreover, using the established Boxing inequality, we derive several fractional Poincaré-type inequalities; notably, these are the first such inequalities to be proved with the Bourgain--Brezis--Mironescu (BBM) factor $1-\delta$  on John domains, and are new even for Lipschitz domains.
 
\subsection{{John domains}} {A bounded domain $\Omega\subset \IR^n$ is an {\it $s$-distance John domain}, $s\geq 1$, if there is a constant $C_J\geq 1$ and a distinguished point $x_J\in \Omega$, so that  each point $x\in \Omega$ can be joined to $x_J$ by a  rectifiable curve $\gamma: [0, 1]\rightarrow \Omega$, $\gamma(0)=x$, $\gamma(1)=x_J$ satisfying
\begin{equation}\label{def-John}
	C_J d(\gamma(t), \partial \Omega)> |\gamma(t)-x|^{s}
\end{equation}
for all $t\in [0, 1]$. The distinguished point $x_J$ is called the {\it John center} of $\Omega$ and the rectifiable curve $\gamma$ is called the {\it John curve} that joins $x$ with $x_J$. The constant $C_J$ in \eqref{def-John} is called the {\it John constant}. 
If we replace the right-hand side of \eqref{def-John} with the diameter (resp. the length) of the subcurve $\gamma[0, t]$, then we call \(\Omega\) an {\it $s$-diameter John domain} (resp. {\it $s$-length John domain}). }

By definition, it is clear that an $s$-length John domain is $s$-diameter John and an $s$-diameter John domain is $s$-distance John. If $s=1$, the converse implications hold and the domain is simply called a $1$-John domain; see \cite[Theorem 2.16]{NV91}. However, for $s>1$, the converse implications do not hold in general; see \cite{Guo-Koskela} and \cite[Section 6]{Guo17}.

John domains were introduced by Martio--Savas \cite{MS79}, named after F. John \cite{John61} who used this condition in studying elasticity. The term $s$-John domain, $s\geq 1$, was introduced by Smith--Stegenga \cite{SS90}. It is worth noting that smooth domains and  Lipschitz domains are $1$-John domains (cf. \cite[page 435]{HK98}) and cusp domains are $s$-length John domains for suitable $s$  {(cf. \cite[page 435]{HK98})}.

\subsection{Boxing inequality on John domains}  We first introduce the following notion of  {\it improved  relative fractional perimeter}. Let $\Omega\subset \mathbb R^n$ and let $\delta, \tau\in(0, 1)$.  For any Lebesgue measurable set $U\subset \Omega$, define 
$$\mathcal P_{\delta, \tau}(U, \Omega):=\int_\Omega \int_{\left\lvert x-y\right\rvert <\tau \dist(y,\partial\Omega)}\frac{\left\lvert \chi_U(x)-\chi_U(y)\right\rvert }{\left\lvert x-y\right\rvert^{n+\delta} }\,dx\,dy =[\chi_U]_{ \mathcal {W}^{\delta,1}(\Omega)},$$
where $[\chi_U]_{ \mathcal {W}^{\delta,1}(\Omega)}$ is the improved fractional Sobolev energy of $\chi_U$. See Section \ref{FP} for a short introduction of (improved) fractional Sobolev spaces and (improved) fractional perimeters and their relations. It is clear from the definition that $\mathcal P_{\delta, \tau}(U, \Omega)\leq P_\delta(U, \Omega)$.

The main goal of this paper is to establish the Boxing inequality on $s$-distance John domains and to broaden its scope of application. The result is stated as follows.

{
	\begin{thm}\label{geometric-boxing}
			Let $\delta, \tau, \gamma \in (0, 1)$. Assume that $\Omega\subset \IR^n$ is an $s$-distance John domain with $1\leq s\leq \frac{n}{n-\delta}$. Then there exists a positive constant $C=C(n, \tau, s, \gamma, C_J, \diam \Omega)$  such that for every  Lebesgue measurable set $U\subset \Omega$ with $|U|/|\Omega|\leq \gamma$,   one can find  a Lebesgue measure zero set $\mathcal{N}_{U}$ and a covering $\{B(x_i,r_i)\}_{i=1}^{\infty}$ of $U\setminus\mathcal{N}_{U}$ for which
		\begin{align}\label{intro-geometric-boxing}
				\sum_{i=1}^{\infty}r_i^{s(n-\delta)} \le C (1-\delta)\mathcal P_{\delta, \tau}(U,\Omega).
		\end{align} 
If  $U$ is an open set, then $\mathcal{N}_{U}=\emptyset$.
	\end{thm}
}

If $\Omega$ is a smooth domain (hence a $1$-John domain) and $U\subset\Omega$ is open, then $\mathcal{N}_U=\emptyset$. In this case, since $\mathcal{P}_{\delta,\tau}(U,\Omega)\le P_\delta(U,\Omega)$, the Boxing inequality \eqref{frac-boxing-local} follows immediately from \eqref{intro-geometric-boxing}. We note that Ponce--Spector proved \eqref{frac-boxing-local} by contradiction without quantifying the constant $C$ in \eqref{frac-boxing-local}. We obtain a quantitative version of $C$ in \eqref{intro-geometric-boxing}.
In particular, since in this case the constant $C$ in \eqref{intro-geometric-boxing} is $C(n,\tau,\gamma,C_J)$ (see Section \ref{John and Lip} for the Boxing inequality on $1$-John domains), this reveals that
the constant $C$ in \eqref{frac-boxing-local} depends on $\Omega$ only through the John constant $C_J$.

In terms of the Hausdorff content $\mathcal H_\infty^{s(n-\delta)}$ (see Section \ref{Hausdorff}), the Boxing inequality \eqref{intro-geometric-boxing} is equivalent to
\begin{equation}\label{content-boxing}
	\mathcal H_\infty^{s(n-\delta)}\big(U\setminus\mathcal{N}_{U}\big) \le C (1-\delta)\mathcal P_{\delta, \tau}(U,\Omega). 
\end{equation}
Moreover, when $U$ is not open, it is necessary to have the Lebesgue measure zero set $\mathcal{N}_{U}$ in the Boxing inequalities \eqref{intro-geometric-boxing} and \eqref{content-boxing}; see the following example.

 \begin{exam}\label{exam-geometric}
Let $\Omega=(0,1)^n$ with $n\geq 2$. Then $\Omega$ is clearly a $1$-John domain. For any $\delta\in (0,1)$, there exists a compact Cantor-like set $U$ in  $\Omega$ such that $0<\mathcal H^{n-\delta}(U)<\infty$ (cf. \cite[Section 4.12]{Mattila}).  Since the Hausdorff content \(\mathcal{H}_\infty^{n-\delta}\) and the Hausdorff measure \(\mathcal{H}^{n-\delta}\) have the same negligible sets (cf. \cite[Exercise 8.6]{Heinonen}), we have $\mathcal H^{n-\delta}_{\infty}(U)>0$. However, $\mathcal{H}^{n-\delta}(U)<\infty$ implies $|U|=0$ and hence $P_{\delta,\tau}(U,\Omega)=0$. 
This shows that if one does not remove the Lebesgue measure zero set $\mathcal{N}_U$, the Boxing inequalities \eqref{intro-geometric-boxing} and \eqref{content-boxing} fail to hold in general.
\end{exam}

The Boxing inequality has powerful implications for the study of Sobolev functions, as can be seen from its functional formulation. To this end, one defines the Choquet integral of a function $f:\Omega\rightarrow \mathbb R$ with respect to the Hausdorff content $\mathcal H^{s(n-\delta)}_\infty$ {(for more details, see Section \ref{Hausdorff})} by 
\[
\int_{\Omega} |f| \, d\mathcal{H}_\infty^{ s(n-\delta)} := \int_0^\infty \mathcal{H}_\infty^{s(n-\delta)}(\{ |f| > t \}) \, dt.
\]
Then we state our functional formulation of the Boxing inequality.

	\begin{thm}\label{functional-boxing}
	Let $\delta, \tau \in (0, 1)$. Assume that $\Omega\subset \IR^n$ is an $s$-distance John domain with $1\leq s\leq \frac{n}{n-\delta}$. Then there exists a positive constant $C=C(n, \tau, s,  C_J, \diam \Omega)$  such that for every $f\in  L^1(\Omega)$, 
	\begin{align}\label{intro-func}
		\int_{\Omega}\left\lvert f^*-\vint{\Omega} f\right\rvert \,d\mathcal{H}_{\infty} ^{s(n-\delta)} \leq C(1-\delta)\int_\Omega \int_{\left\lvert x-y\right\rvert <\tau \dist(y,\partial\Omega)}\frac{\left\lvert f(x)-f(y)\right\rvert }{\left\lvert x-y\right\rvert^{n+\delta} }\,dx\,dy.
	\end{align}
\end{thm}
Here $f^*$ in \eqref{intro-func} is the precise representative of $f$ (see Section \ref{Chapter-4} for details) and $\vint{\Omega} f = \frac{1}{|\Omega|}\int_\Omega f\,dx$. If $f$ is continuous, then $f^* = f$ and \eqref{intro-func} holds with $f$ in place of $f^*$. For general $f \in L^1(\Omega)$, the inequality must be formulated using
the precise representative $f^*$. Indeed, let $\Omega$ and $U$ be the sets in Example \ref{exam-geometric} and take $f = \chi_U$. Then it is clear that $\vint{\Omega} f = 0$ and 
\[
\int_{\Omega}\left| f - \vint{\Omega} f \right| \, d\mathcal{H}_\infty^{n-\delta} = \mathcal{H}_\infty^{n-\delta}(U) > 0.
\]
In contrast, the right-hand side of \eqref{intro-func} equals $C(1-\delta)\mathcal P_{\delta,\tau}(U,\Omega)$, which is $0$. This shows that if one replaces $f^*$ with $f$, the Boxing inequality \eqref{intro-func} fails to hold in general.
On the other hand, 
for $f = \chi_U$ we have $f^* \equiv 0$ and hence \eqref{intro-func} holds. 

\begin{rem}
The Boxing inequality in Theorem \ref{geometric-boxing} is referred to as the geometric Boxing inequality and the one in Theorem \ref{functional-boxing} as the functional Boxing inequality. These two Boxing inequalities can be shown to be  equivalent to each other; see Proposition \ref{prop-equivalence}.
\end{rem}

\subsection{Characterizations of John domains} Up until now, we have considered Boxing inequalities on John domains. 
It turns out that a converse direction also holds; we show that John domains are essentially the only domains supporting such inequalities.  To this end, we first introduce
a mild geometric condition called the separation property.

A domain $\Omega \subset \mathbb{R}^n$  is said to have the {\it separation property} if there exist  a distinguished point $x_0\in \Omega$ and a constant $C_0\geq 1$ such that the following property holds: For every $x \in \Omega$, there exists a curve $\gamma : [0,1] \to \Omega$ with $\gamma(0) = x$, $\gamma(1) = x_0$, such that for each $t$, either
\[
\gamma([0, t]) \subset B_t := B(\gamma(t), C_0 d(\gamma(t), \partial\Omega))
\]
or each $y \in \gamma([0, t]) \setminus B_t$ and $x_0$ belongs to different components of $\Omega \setminus \partial B_t$. A simply connected domain on the plane automatically has the separation property. In $\mathbb R^n$ with $n\geq 3$, any domain that is quasiconformally equivalent to a uniform domain has the separation property. We refer to \cite{BP95} for more details.  

Assuming the separation property, Buckley--Koskela \cite{BP95} showed that a bounded domain supporting the Sobolev--Poincaré inequality is a John domain. Subsequently, Guo \cite[Theorem 1.4]{Guo17} extended this result to bounded domains supporting fractional Sobolev--Poincaré inequalities.

In the following, we show that the same conclusion remains valid for bounded domains supporting the Boxing inequality.
\begin{thm}\label{Box-John}
Let $\Omega\subset\IR^n$ be a bounded domain that satisfies the separation property. Let $\delta,\tau\in (0,1)$ and $n-\delta\leq \beta< n$. Suppose that $\Omega$ supports the following  Boxing inequality:  for any $\gamma\in (0,1)$, there exists a constant $C>0$ such that for every open set $U\subset \Omega$ with $|U|/|\Omega|\leq \gamma$, 
\begin{equation}\label{box-in-john}
\mathcal{H}_\infty^{\beta} (U)\leq C(1-\delta) \mathcal P_{\delta, \tau}(U,\Omega).
\end{equation}
Then $\Omega$ is an $s$-diameter John domain with $s=\delta\beta/\big((n-\delta)(n-\beta)\big)$.
\end{thm}

The following result considers the special case $\beta=n-\delta$. Combining Theorem \ref{geometric-boxing} with Theorem \ref{Box-John} establishes the equivalence between the $1$-John property and the Boxing inequality, indicating that John domains are essentially the only
domains supporting the Boxing inequality.

\begin{cor}\label{Box-relative}
Let $\Omega\subset\IR^n$ be a bounded domain that satisfies the separation property, and let $\delta, \tau\in (0,1)$. Then the following two conditions are equivalent.
\begin{itemize}
\item[(1)] $\Omega$ is a $1$-John domain.	
	
\item[(2)] $\Omega$ supports the Boxing inequality \eqref{intro-geometric-boxing} with $s(n-\delta)$ replaced by $n-\delta$.

\end{itemize}
\end{cor}

\subsection{Application to fractional Poincar\'e-type inequalities}

As an application of the Boxing inequality established above, we derive several fractional Poincar\'e-type inequalities. To state our first result, recall that for a given Radon measure $\mu$, 
$
\mu \le C \mathcal H_\infty^\beta
$
is equivalent to the condition that 
$
\mu(B(x,r)) \le C r^\beta
$
holds for every $B(x,r)\subset \mathbb R^n$, where $0<\beta\le n$ (see Section~\ref{sec:prelis}). Thus, we can directly derive a corollary of Theorem~\ref{functional-boxing}, which we call the \emph{fractional Poincar\'e--Wirtinger trace inequality}, from which both the fractional Hardy-type inequality and the fractional Sobolev--Poincar\'e inequality follow as consequences.
\begin{cor}\label{cor:1-john-Trace-intro}
	Let $\delta, \tau\in (0, 1),$ $1\leq s\leq \frac{n}{n-\delta}$, let $\Omega\subset \IR^n$ be an $s$-distance John domain and let $\mu$ be a Radon measure on $\IR^n$. Assume that there exists a constant $C_\mu\geq 1$ such that 
	\begin{align}\label{eq:growth-estimate-intro-1}
		\mu(B(x,r))\leq C_\mu r^{s(n-\delta)}
	\end{align}
	for any $B(x,r)\subset \IR^n$. Then there exists  a positive constant $C=C(n, \tau, s,C_\mu, C_J, \diam \Omega)$
	such that for every $f\in L^1(\Omega)$,
	\begin{align}\label{eq:fist-cor}
		\left\| f^{*} - \vint\Omega f \right\|_{L^1(\Omega, d\mu)}\leq C(1-\delta)\int_\Omega \int_{|x - y| < \tau \dist(y,\partial\Omega)}\frac{\left\lvert f(x)-f(y)\right\rvert }{\left\lvert x-y\right\rvert^{n+\delta} }\,dx\,dy.
	\end{align}
\end{cor}

The Poincar\'e--Wirtinger trace inequality was first introduced by Meyers--Ziemer \cite[Theorem 4.10]{MZ77}. Subsequently, Ponce--Spector \cite[Theorem 1.6]{PS20} and Myyryläinen--Pérez--Weigt \cite[Corollary 4.3]{MPW24} studied such inequalities in the fractional setting on smooth domains or cubes, incorporating the BBM factor $1-\delta$. Corollary \ref{cor:1-john-Trace-intro} extends their results to the class of $s$-distance John domains. 

Moreover, with the aid of the BBM factor $1-\delta$, and applying the Bourgain--Brezis--Mironescu formula on arbitrary domains (cf. \cite[Theorem 1]{M24}), which states that for $f\in W^{1,1}(\Omega)$,
\begin{align}\label{intoBBM-ar}
	\lim_{\delta\rightarrow 1^-} (1 - \delta) \int_{\Omega} \int_{|x - y| < \tau \operatorname{dist}(y, \partial \Omega)} \frac{|f(y) - f(x)|}{|x - y|^{n + \delta}} \, dx\,dy = C(n) \int_{\Omega} |Df(x)| \,dx,
\end{align}
to the fractional Poincar\'e--Wirtinger trace inequality in Corollary \ref{cor:1-john-Trace-intro}, we obtain, after some nontrivial effort, the following Poincar\'e--Wirtinger trace inequality on $s$-distance John domains. This extends the classical result of Meyers--Ziemer \cite[Theorem 4.10]{MZ77}. For $f\in L^{1}(\Omega),$ let $\| Df \|(\Omega)$ denote the total variation of $f$ in $\Omega$; see Section \ref{Chapter-6}.

\begin{cor}\label{intro:s-john-gFPI}
Let $1\leq s\leq \frac{n}{n-1}$, let $\Omega\subset \IR^n$ be an $s$-distance John domain and let $\mu$ be a Radon measure on $\IR^n$. Assume that there exists a constant $C_\mu\geq 1$ such that 
	\begin{align}\label{eq:growth-estimate-intro-2}
		\mu(B(x,r))\leq C_\mu r^{s(n-1)}
	\end{align}
	for any $B(x,r)\subset \IR^n$. Then there exists  a positive constant $C=C(n, s, C_\mu, C_J, \diam \Omega)$ such that for every $f\in L^{1}(\Omega)$,
	\begin{align}\label{introeq:smooth-1}
		\left\lVert f^{*}-\vint{\Omega}f\right\rVert _{L^{1}(\Omega,d\mu)} \leq C \left\lVert Df\right\rVert (\Omega). 
	\end{align}
\end{cor}
By a duality argument, Corollary \ref{cor:1-john-Trace-intro} and Corollary \ref{intro:s-john-gFPI} imply the following $L^q$-versions. 
\begin{itemize}
\item  Let the assumptions of Corollary~\ref{cor:1-john-Trace-intro}
hold, except that the growth condition \eqref{eq:growth-estimate-intro-1} on $\mu$ is replaced by
\begin{equation}\label{q-measure}
 \mu(B(x,r))\le C_\mu r^{qs(n-\delta)}
\end{equation}
for every ball $B(x,r)\subset\IR^n$, where
$1\le q\le n/(s(n-\delta))$. Then, for every $f\in L^1(\Omega)$,
\begin{equation}\label{cor:1-john-Trace-intro-q}
\left\|f^*-\vint{\Omega} f\,dx\right\|_{L^q(\Omega,d\mu)}
\le C(1-\delta)
\int_\Omega\int_{|x-y|<\tau\operatorname{dist}(y,\partial\Omega)}
\frac{|f(x)-f(y)|}{|x-y|^{n+\delta}}\,dx\,dy,
\end{equation}
where $C=C(n,\tau,s,C_\mu,C_J,\operatorname{diam}\Omega)$.

\item Let the assumptions of Corollary~\ref{intro:s-john-gFPI}
hold, except that the growth condition \eqref{eq:growth-estimate-intro-2} on $\mu$ is replaced by
\[
    \mu(B(x,r))\le C_\mu r^{qs(n-1)}
\]
for every ball $B(x,r)\subset\IR^n$, where
$1\le q\le n/(s(n-1))$. Then, for every $f\in L^1(\Omega)$,
\begin{equation}\label{eq:growth-estimate-intro-2q}
\left\|f^*-\vint{\Omega} f\,dx\right\|_{L^q(\Omega,d\mu)}
\le C\|Df\|(\Omega),
\end{equation}
where $C=C(n,s,C_\mu,C_J,\operatorname{diam}\Omega)$.
\end{itemize}

We give the argument only for the fractional inequality
\eqref{cor:1-john-Trace-intro-q}.
Inequality \eqref{eq:growth-estimate-intro-2q} follows from
Corollary~\ref{intro:s-john-gFPI} by the same duality argument, so we omit the details. For $1< q\le n/(s(n-\delta))$, let $q'$ be the conjugate exponent of $q$. By the duality of $L^q$ spaces,
\[
\left\| f^* - \vint{\Omega} f \right\|_{L^q(\Omega, d\mu)}
= \sup_{\|g\|_{L^{q'}(\mu)} \le 1}
\left| \int_\Omega \left( f^* - \vint{\Omega} f \right) g \, d\mu \right|.
\]
For each $g \in L^{q'}(\mu)$ with $\|g\|_{L^{q'}(\mu)}\le 1$, define $\tilde{\mu}=|g|\,d\mu$. If $\mu$ satisfies \eqref{q-measure}, then for any ball $B(x,r)$, H\"older inequality yields
\[
\tilde{\mu}(B(x,r))
= \int_{B(x,r)} |g| \, d\mu
\le \left( \int_{B(x,r)} |g|^{q'} \, d\mu \right)^{1/q'}
\left( \mu(B(x,r)) \right)^{1/q}
\le C_\mu^{1/q} r^{s(n-\delta)},
\]
where we have used $\|g\|_{L^{q'}(\mu)}\le 1$. Thus $\tilde{\mu}$ satisfies the  growth condition \eqref{eq:growth-estimate-intro-1}. Applying Corollary \ref{cor:1-john-Trace-intro} to $\tilde{\mu}$ gives
\[
\int_\Omega \left| f^* - \vint{\Omega} f \right| |g| \, d\mu
\le C_0 (1-\delta)
\int_\Omega \int_{|x-y| < \tau \operatorname{dist}(y,\partial\Omega)}
\frac{|f(x)-f(y)|}{|x-y|^{n+\delta}} \, dx\,dy,
\]
where $C_0 = C_0(n, \tau, s, C_\mu^{1/q}, C_J, \operatorname{diam}\Omega)$. Taking the supremum over all such $g$ and absorbing $C_\mu^{1/q}$ into $C_\mu$ (since $q\ge 1$ and $C_\mu\ge 1$) yields \eqref{cor:1-john-Trace-intro-q} with constant
$C=C(n, \tau, s, C_\mu, C_J, \operatorname{diam}\Omega)$.


Some more results concerning special Radon measures $\mu$ are presented.
First, in Corollary \ref{cor:1-john-Trace-intro}, assume that $d\mu = {dx}/{|x|^{n-s(n-\delta)}}$. A direct computation shows that for any  \(B(x,r) \subset \mathbb{R}^n\),
\[
\int_{B(x,r)}\frac{1}{|y|^{n-s(n-\delta)}}\,dy
\leq
\int_{B(0,r)}\frac{1}{|y|^{n-s(n-\delta)}}\,dy
=
\frac{n\omega_{n}}{s(n-\delta)}r^{s(n-\delta)}\leq \frac{n\omega_{n}}{n-1}r^{s(n-\delta)},
\]
where $\omega_n$ is the volume of the unit ball in $\mathbb{R}^n$. This shows that such measure $\mu$ satisfies \eqref{eq:growth-estimate-intro-1}, and hence 
 \eqref{eq:fist-cor}   yields the following {\it fractional Hardy-type inequality}. To the best of our knowledge, this inequality first appeared in Maz'ya--Shaposhnikova \cite{MS02}. Its original proof is valid for the cases of \(\mathbb{R}^n\) and cubes. 
Here we generalize the result to $s$-John domains, 
of which cubes are  special cases.
\begin{cor}\label{cor:1-john-hardy-intro}
Let $\delta, \tau \in (0, 1),$ $1\leq s\leq \frac{n}{n-\delta}$. Suppose that $\Omega\subset \IR^n$ is an $s$-distance John domain.
	There exists  a constant $C=C(n, \tau, s, C_J,\diam \Omega)$ such that the inequality
	\begin{align}
			\int_{\Omega}|f(x)|\frac{\,dx}{|x|^{n-s(n-\delta)}}\leq C(1-\delta)\int_\Omega \int_{|x - y| < \tau \dist(y,\partial\Omega)}\frac{\left\lvert f(x)-f(y)\right\rvert }{\left\lvert x-y\right\rvert^{n+\delta} }\,dx\,dy
	\end{align}
	holds for every $f\in L^{1}(\Omega)$ with $\vint\Omega f =0.$
\end{cor}

On the other hand, in \eqref{cor:1-john-Trace-intro-q}, choosing $q = \frac{n}{s(n-\delta)}$ and taking $\mu$ to be the Lebesgue measure  recovers the classical fractional Sobolev--Poincar\'e inequality, which contains the BBM factor $1-\delta$ and extends the result of Bourgain--Brezis--Mironescu \cite[Theorem 1]{BBM02} to $s$-distance John domains. We state this as follows.
 
 \begin{cor}\label{intro-s-john-FSPI}
 	Let $\delta, \tau\in (0, 1)$, $1\leq s\leq \frac{n}{n-\delta}$ and let $\Omega\subset \IR^n$ be an $s$-distance John domain. Then there exists  a positive constant $C=C(n, s, \tau,C_J, \diam \Omega)$ such that for every $f\in L^{1}(\Omega)$,
 	\begin{align}\label{intro-s-john-FSPI-eq}
 		\left\| f - \vint{\Omega} f \right\|_{L^{\frac{n}{s(n-\delta)}}(\Omega)}\leq C(1-\delta)\int_\Omega \int_{\left\lvert x-y\right\rvert <\tau \dist(y,\partial\Omega)}\frac{\left\lvert f(x)-f(y)\right\rvert }{\left\lvert x-y\right\rvert^{n+\delta} }\,dx\,dy.
 	\end{align}
 \end{cor}

We note that the inequality \eqref{intro-s-john-FSPI-eq} without the BBM factor $1-\delta$ was previously established by Guo \cite[Theorem 1.2]{Guo17}; the exponent $n/(s(n-\delta))$ on the left-hand side is sharp in view of \cite[Example 1.1]{Guo17}. Likewise, in \eqref{eq:growth-estimate-intro-2q}, taking $q = \frac{n}{s(n-1)}$ and choosing $\mu$ as the Lebesgue measure yields the classical Sobolev--Poincar\'e inequality (with $p=1$) of Haj\l asz--Koskela \cite[Corollary 5]{HK98} on $s$-distance John domains.
Moreover, we not only recover the earlier theorems of Haj\l asz--Koskela \cite[Corollary 5]{HK98} and Guo \cite[Theorem 1.2]{Guo17}, both originally established on $s$-length John domains, but also extend them to the broader setting of $s$-distance John domains.

Finally, for special domains such as $1$-John domains and Lipschitz domains, all the above inequalities remain valid with constants independent of $s$ and $\operatorname{diam}\Omega$, since in these cases $s=1$, the inequalities are scaling invariant. We refer to Section~\ref{John and Lip} for a detailed discussion.

\textbf{Structure.} The structure of the paper is as follows.
In Section~\ref{sec:prelis}, we introduce the notation used in this paper and recall some basic facts on Hausdorff content, Hausdorff measure, Choquet integral, fractional Sobolev spaces and fractional perimeter.
In Section~\ref{sec-3}, we prove the functional Boxing inequality for continuous functions on $s$-distance John domains.
In Section~\ref{Chapter-4}, we study the density property of improved fractional Sobolev spaces and the Hausdorff dimension estimate of the exceptional set of $f$. These results are crucial for applying the results from Section \ref{sec-3} to $L^1$-functions.
In Section~\ref{sec-5}, we show the functional Boxing inequality for $L^1$-functions on $s$-distance John domains (Theorem~\ref{functional-boxing}) and derive the geometric Boxing inequality (Theorem~\ref{geometric-boxing}) from Theorem~\ref{functional-boxing}. We also prove the equivalence between the functional and geometric Boxing inequalities (Proposition~\ref{prop-equivalence}).
In Section~\ref{Chapter-6}, we establish the Poincaré--Wirtinger inequality (Corollary~\ref{intro:s-john-gFPI}). Moreover, under the  separation property, we prove Theorem~\ref{Box-John}.
Finally, in Section~\ref{John and Lip}, we summarize all the results for $1$-John domains and Lipschitz domains.

\medskip


\section{\large Preliminaries}\label{sec:prelis}
Throughout this paper, let $\mathbb{N} = \{1,2,\dots\}$ and let $\Omega\subset\IR^n$ (with $n\geq 2$) be a domain (namely, an open and connected set) which is nonempty. We denote the standard Euclidean norm of a point $x\in\mathbb{R}^n$ by $\lvert x \rvert$, and denote the diameter of $\Omega$ by $\diam \Omega:=\sup\{|x-y|: x, y\in \Omega\}$.
The Lebesgue measure of a measurable subset  $U$ of $\mathbb{R}^n$ is denoted by $|U|$. 
The distance from a point
$x$ to the boundary $\partial \Omega$ of $\Omega$ is $\dist(x,\partial \Omega):=\inf_{y\in\partial\Omega}|x-y|.$ In cases where no confusion arises, we often abbreviate $d(x)=d(x, \partial\Omega)$. The open ball with center $x\in \mathbb R^n$ and radius $r>0$ is denoted by $B(x, r)$. Let $\omega_n$ denote the volume of the unit ball in $\mathbb{R}^n$. Then the surface area of the unit sphere in $\mathbb{R}^n$ is $n\omega_n$.

Constants are positive and denoted by $C$; we write $C=C(a,b,\ldots)$ to signify that the constant
depends only on the parameters $a,b,\ldots$.

Assume that $\mu$ is a Radon measure on $\Omega$ and $1\leq q<\infty$;
we only work with nonnegative measures.
The space $L^q(\Omega, \mu)$ consists of all $\mu$-measurable functions $f$ on $\Omega$ whose $L^q$-norm  is finite. The $L^q$-norm of $f$ is given by
\[
\left\| f  \right\|_{L^q(\Omega, d\mu)} :=\left(\int_{\Omega}|f|^q\,d\mu\right)^{1/q}.
\]
When $\mu$ is the Lebesgue measure, we simply write $L^q(\Omega)$ instead of $L^q(\Omega, \mu)$. For any $f\in L^1(\Omega)$, 
assuming that $0<|\Omega|<\infty$,
the integral average of $f$ over $\Omega$
is denoted by
\[
f_\Omega :=\vint{\Omega}f\,dx := \frac{1}{|\Omega|} \int_{\Omega}f\,dx .
\]

\begin{lem}\label{lem:q-integral}
	Let $\mu$ be an outer measure on a domain $\Omega\subset \mathbb R^n$ and let $1\leq q<\infty$. Then for every function $f$ defined on $\Omega$, we have
	\[
	\left(\int_0^\infty\lambda^{q-1}\mu(\Omega_\lambda)\, d\lambda\right)^{1/q} \leq  \int_{0}^{\infty}\mu(\Omega_\lambda)^{{1}/{q}}\,d\lambda,
	\]
	where $\Omega_\lambda=\left\{x\in\Omega:\left\lvert f\right\rvert>\lambda  \right\}$ with $\lambda\geq 0$.
\end{lem}
\begin{proof}
	Since $\Omega_\lambda\subset\Omega_t$ for $0<t\leq \lambda,$ it holds that 
	$$\lambda^{q-1}\mu(\Omega_\lambda)^{\frac{q-1}{q}}=\left(\int_{0}^{\lambda}\mu(\Omega_\lambda)^{1/q}\,dt\right)^{q-1}\leq \left(\int_{0}^{\lambda}\mu(\Omega_t)^{1/q}\,dt\right)^{q-1}.$$
	This implies
	\begin{align*}
		\left(\int_0^\infty\lambda^{q-1}\mu(\Omega_\lambda)\, d\lambda\right)^{1/q}	&=\left(\int_{0}^{\infty}\lambda^{q-1}\mu(\Omega_\lambda)^{\frac{q-1}{q}}\mu(\Omega_\lambda)^{1/q}\,d\lambda\right)^{1/q}\\
		&\leq \left(\int_{0}^{\infty}\left(\int_{0}^{\lambda}\mu(\Omega_t)^{1/q}\,dt\right)^{q-1}\mu(\Omega_\lambda)^{1/q}\,d\lambda\right)^{1/q}\\
		&\leq \left(\int_{0}^{\infty}\mu(\Omega_t)^{1/q}\,dt\right)^{\frac{q-1}{q}}\left(\int_{0}^{\infty}\mu(\Omega_\lambda)^{1/q}\,d\lambda\right)^{1/q}\\
		&=\int_{0}^{\infty}\mu(\Omega_\lambda)^{{1}/{q}}\,d\lambda,
	\end{align*}
	which finishes the proof.
\end{proof}

\subsection{Hausdorff content, Hausdorff measure and Choquet integral} \label{Hausdorff}
We give a brief introduction about  the definitions and basic properties of Hausdorff content, Hausdorff measure and Choquet integrals. For more details, we refer the interested reader to \cite{Adams98,EGbook,YY08,HH23}.

Let $0<\alpha\leq n, 0<s\leq \infty$.  For every set $ U \subset \mathbb{R}^n$, we define
\[
\mathcal{H}_s^{ \alpha}(U) := \inf\left\{ \sum_{i=0}^\infty \omega_{ \alpha} r_i^{\alpha} : U \subset \bigcup_{i=0}^\infty B(x_i,r_i),\, r_i<s \right\},
\]
where $ \omega_{ \alpha} := \pi^{ \alpha/2} / \Gamma\left( \frac{ \alpha}{2} + 1 \right)$. When $s=\infty$, we call $\mathcal{H}_\infty^{ \alpha}(U)$ the {\it $\alpha$-dimensional  Hausdorff content} of $U$. Moreover, we define {\it  $\alpha$-dimensional Hausdorff measure} of $U$ by
\[
\mathcal{H}^\alpha(U):=\lim_{s\to 0}\mathcal{H}_s^{ \alpha}(U)=\sup_{s>0}\mathcal{H}_s^{ \alpha}(U).
\]
Then, the {\it Hausdorff dimension} of a set $U\subset\IR^n$ is
\[
\text{dim}_{H}(U):=\inf \left\{\alpha:\mathcal{H}^\alpha(U)=0\right\} .
\]
A notable fact concerning the Hausdorff content \(\mathcal{H}_\infty^{\alpha}\)  and the Hausdorff measure \(\mathcal{H}^{\alpha}\) is that they have the same negligible sets; \cite[Exercise 8.6]{Heinonen}. For any ball $B(x, r)$, we have $\mathcal{H}_\infty^{ \alpha}(B(x, r))=\omega_\alpha r^\alpha. $ Therefore, if $U\subset \mathbb R^n$ is a Lebesgue measurable set, we have $|U|=\mathcal H^{n}_\infty(U)=\mathcal H^{n}(U)$; see \cite[Theorem 2.5]{EGbook} and the proof of \cite[Proposition 2.5]{HH23}

Let $\mu$ be a Radon measure on $\mathbb R^n.$ Suppose there exists a constant $C_\mu$ such that for any $B(x, r)$, 
\begin{equation}\label{eq:mu-H}
\mu(B(x, r))\leq C_\mu r^\alpha.
\end{equation}
Then from the definition of the Hausdorff content,
there is a constant
{$C=C(n,C_\mu)$ such that for any $\mu$-measurable set $U\subset \mathbb R^n$ and for all $0\le \alpha\le n$,}
\begin{equation}\label{mu-content}
\mu(U)\leq C \mathcal H^\alpha_\infty(U).
\end{equation}
In fact, for a Radon measure $\mu$, \eqref{eq:mu-H} and \eqref{mu-content} are equivalent.

The {\it dyadic Hausdorff content} $\widehat {\mathcal H}_{\infty}^{\alpha}$ is defined for any set $U\subset \mathbb R^n$ by
\[
\widehat {\mathcal{H}}_\infty^{ \alpha}(U) := \inf\left\{ \sum_{i=1}^\infty \ell_i^{\alpha}  : U \subset \text{int}\,\bigcup_{i=1}^\infty Q_i \right\},
\]
where the infimum is taken over all sequences of closed dyadic cubes $Q_i$ and $\ell_i$ is the side
length of the cube $Q_i$. It follows from \cite[Proposition 2.3]{YY08} that there exists a positive
{constant $C=C(n)$ such that for any set $U\subset \mathbb R^n$
and for all $0\le \alpha\le n$,}
\begin{equation}\label{dyadic-comparable}
	C^{-1}\mathcal{H}_\infty^{ \alpha}(U)\leq \widehat{\mathcal{H}}_\infty^{ \alpha}(U)\leq C \mathcal{H}_\infty^{ \alpha}(U).
\end{equation}

The Hausdorff content $\mathcal{H}_\infty^{ \alpha}$ is 
{not known} 
to be strongly subadditive. 
However, the dyadic Hausdorff content $\widehat{\mathcal{H}}_\infty^{ \alpha}$ is strongly
subadditive in the sense that 
\begin{equation}\label{strong subadditive}
	\widehat{\mathcal{H}}_\infty^{ \alpha}(U_1\cup U_2)+\widehat{\mathcal{H}}_\infty^{ \alpha}(U_1\cap U_2)\leq \widehat{\mathcal{H}}_\infty^{ \alpha}(U_1)+\widehat{\mathcal{H}}_\infty^{ \alpha}(U_2)
\end{equation}
for every $U_1, U_2\subset \mathbb R^n$.
Using the strong subadditivity, one can show that
the dyadic Hausdorff content $\widehat{\mathcal{H}}_\infty^{ \alpha}$ is a capacity in the sense of Choquet; see \cite[Theorem 2.1]{YY08}. Thus, $\widehat{\mathcal{H}}_\infty^{ \alpha}$ satisfies
\begin{equation}\label{increasing-lemma}
	\lim_{k\rightarrow \infty} \widehat{\mathcal{H}}_\infty^{ \alpha}(U_k) =\widehat{\mathcal{H}}_\infty^{ \alpha}\left(\bigcup_{k=1}^{\infty} U_k\right)
\end{equation}
for every non-decreasing sequence $\{U_k\}_{k=1}^{\infty}$ of subsets of $\mathbb R^n$. 

For any function $ f: \Omega \to \mathbb{R}$,
the Choquet integral of $f$
with respect to the Hausdorff content $\mathcal{H}_\infty^{ \alpha}$ is defined by
\[
\int_{\Omega} |f| \, d\mathcal{H}_\infty^{ \alpha} := \int_0^\infty \mathcal{H}_\infty^{\alpha}(\{ |f| > t \}) \, dt.
\]

We list the following basic properties of the Choquet integral, which can be found in  \cite{Adams98,HH23} and will be used without further mention later. For any two nonnegative functions $f$ and $g$ on $\Omega$, we have the following.

\begin{description}
\item[(Monotonicity)] $\displaystyle\int_{\Omega} f(x) \, d\mathcal{H}_{\infty}^{\alpha} \leq \int_{\Omega} g(x) \, d\mathcal{H}_{\infty}^{\alpha}\ $ whenever $f(x)\leq g(x)$ for $\mathcal H^{\alpha}_\infty$-a.e. $x\in \Omega$.
	\item[(Sublinearity)] $\displaystyle\int_{\Omega} f(x) + g(x) \, d\mathcal{H}_{\infty}^{\alpha} \leq 2\left( \int_{\Omega} f(x) \, d\mathcal{H}_{\infty}^{\alpha} + \int_{\Omega} g(x) \, d\mathcal{H}_{\infty}^{\alpha} \right).$
	\item[(H\"older's inequality)] Let $p, q\in (1, \infty)$ with $1=1/p+1/q$. Then $$\displaystyle\int_{\Omega} f(x)g(x) \, d\mathcal{H}_{\infty}^{\alpha} \leq 2\left( \int_{\Omega} f(x)^p \, d\mathcal{H}_{\infty}^{\alpha} \right)^{1/p} \left( \int_{\Omega} g(x)^{q} \, d\mathcal{H}_{\infty}^{\alpha} \right)^{1/q}.$$
\end{description}
Similarly, we can define the integral of $f$ with respect to the dyadic Hausdorff content $\widehat{\mathcal{H}}_{\infty}^\alpha$ and the above basic properties will hold as well.

\subsection{Fractional Sobolev spaces and fractional perimeter}\label{FP}
Let $\delta\in (0, 1)$, $p\in [1, \infty)$ and let $\Omega\subset \mathbb R^n$ be a domain.
The {\it fractional Sobolev space} $W^{\delta, p}(\Omega)$ is defined  as
\[
 {W}^{\delta,p}(\Omega):=\left\{f\in L^{p}(\Omega):~[f]_{ {W}^{\delta,p}(\Omega)}:=\int_{\Omega}\int_{\Omega}\frac{|f(x)-f(y)|^p}{|x-y|^{n+\delta p}}\,dx\,dy<\infty\right\},
\]
with the norm
\[
\left\lVert f\right\rVert^p _{{W}^{\delta,p}(\Omega)}:= \left\| f  \right\|^p_{L^p(\Omega)}+ [f]_{{W}^{\delta,p}(\Omega)}.
\]

When studying the fractional Sobolev--Poincar\'e inequalities on irregular domains (including John domains) and extending the Bourgain--Brezis--Mironescu formula \cite{BBM01} to arbitrary domains, the following improved fractional Sobolev space plays an important role; see \cite{HV13,DyIhVa16,DD22,M24}. Let $\delta, \tau\in (0, 1)$ and $p\in [1, \infty)$, the {\it improved fractional Sobolev space} $\mathcal {W}^{\delta,p}(\Omega)$ is defined as
\[
\mathcal {W}^{\delta,p}(\Omega):=\left\{f\in L^{p}(\Omega):~[f]_{\mathcal {W}^{\delta,p}(\Omega)}:=\int_{\Omega}\int_{|x-y|<\tau  d(y)}\frac{|f(x)-f(y)|^p}{|x-y|^{n+\delta p}}\,dx\,dy<\infty\right\},
\]
with the norm
\[
\left\lVert f\right\rVert^p _{\mathcal {W}^{\delta,p}(\Omega)}:= \left\| f  \right\|^p_{L^p(\Omega)}+ [f]_{\mathcal {W}^{\delta,p}(\Omega)}.
\]
Here, in the improved fractional Sobolev energy $[f]_{\mathcal {W}^{\delta,p}(\Omega)}$, the integral over $|x-y|<\tau d(y)$ means  integration  over the set $\{x\in \Omega: |x-y|<\tau d(y)\}$. Throughout this paper, the notation $|x-y|<\tau d(y)$ always denotes  the set $\{x\in \Omega: |x-y|<\tau d(y)\}$, unless otherwise specified.

It is evident from the definitions that $[f]_{\mathcal W^{\delta, p}(\Omega)}\leq [f]_{W^{\delta, p}(\Omega)}$. Hence, we have the natural embedding $W^{\delta, p}(\Omega)\subset \mathcal W^{\delta, p}(\Omega)$. It is known,
for all $1\le p<\infty$, that the improved fractional Sobolev space $\mathcal W^{\delta, p}(\Omega)$ coincides with $W^{\delta, p}(\Omega)$ (with equivalence of norms) when $\Omega$ is a Lipschitz domain (see \cite[Proposition 5]{Dyda06}), or more generally a uniform domain at least when $1<p<\infty$
(see \cite[Corollary 4.5]{Saksman17}). However, if the domain $\Omega$ is merely a $1$-John domain, then $\mathcal W^{\delta, p}(\Omega)$ and $W^{\delta, p}(\Omega)$
need not coincide for $1<p<\infty$,
as demonstrated by an example constructed in \cite[Proposition 3.4]{DyIhVa16}. To the best of our knowledge, a complete geometric characterization of domains on which these two spaces coincide remains an open problem.

In this paper, we focus primarily on the case
$p=1$. In this case, the (improved) fractional Sobolev spaces  are closely related to fractional perimeters. For any Lebesgue measurable set $U\subset \Omega$, define
$$P_\delta(U, \Omega):=2\int_{U}\int_{\Omega\setminus U}\frac{1}{|x-y|^{n+\delta}}\, dx\, dy.$$
It is evident that 
$$P_\delta(U, \Omega)=\int_{\Omega}\int_{\Omega}\frac{|\chi_{U}(x)-\chi_U(y)|}{|x-y|^{n+\delta}}\, dx\, dy=[\chi_U]_{W^{\delta,1}(\Omega)}.$$
When $\Omega=\mathbb R^n$, one simply writes $P_\delta(U)=P_\delta(U, \mathbb R^n)$. This quantity appears in \cite{PS20} and has been called {\it fractional perimeter} \cite{CaRoSa10} or {\it non-local $\delta$-perimeter} \cite{BLP14}.
The quantity $P_\delta(U, \Omega)$ is usually called  the {\it  relative fractional perimeter}. For a comprehensive introduction, we refer to the survey \cite{Serra24} and references therein.

As indicated above,  the improved fractional Sobolev energy is more suitable when $\Omega$ is a John domain or other irregular domain.  Therefore, we introduce the following notion of an  {\it improved  relative fractional perimeter}. Let $\delta, \tau\in(0, 1)$.  For any Lebesgue measurable set $U\subset \Omega$, define 
$$\mathcal P_{\delta, \tau}(U, \Omega):=[\chi_U]_{ \mathcal {W}^{\delta,1}(\Omega)}=\int_\Omega \int_{\left\lvert x-y\right\rvert <\tau d(y)}\frac{\left\lvert \chi_U(x)-\chi_U(y)\right\rvert }{\left\lvert x-y\right\rvert^{n+\delta} }\,dx\,dy.$$
It is clear that $\mathcal P_{\delta, \tau}(U, \Omega)\leq P_\delta(U, \Omega)$ and $\mathcal P_{\delta, \tau}(U, \mathbb R^n)=P_\delta(U, \mathbb R^n)=P_\delta(U)$. When $\Omega$ is a Lipschitz domain, the relations of the underlying fractional Sobolev spaces imply that $\mathcal P_{\delta, \tau}(U, \Omega)$ is comparable with $P_\delta(U, \Omega)$.

We close this section by presenting the following lemma proved in \cite [Lemma 4.1]{PS20}.
\begin{lem}\label{lem:PS-lem}
	Given $\gamma\in (0,1)$, there exists a constant $C=C(\gamma, n)$ such that 
	\[
r^{n-\delta}\leq C(1-\delta)
	\int_{A}\int_{B(z, r)\setminus A}\frac{dx\,dy}{|x-y|^{n+\delta}}
	\]
	for every $r>0$, every $\delta\in (0,1)$ and every Lebesgue measurable set $A\subset B(z, r)$ such that $|A|/|B(z, r)|=\gamma.$
\end{lem}

The original statement in \cite[Lemma 4.1]{PS20} is given for every Borel set $A\subset B(z, r)$, whereas here we extend it to every Lebesgue measurable set. This extension is immediate because for any Lebesgue measurable set there exists a Borel set such that their symmetric difference has Lebesgue measure zero.


\section{\large {Functional Boxing inequality for continuous functions}}\label{sec-3}

In this section, we derive the functional Boxing inequality for continuous functions on $s$-distance John domains. We start with a weak version of the geometric Boxing inequality.
\begin{prop}\label{prop:sJ-box}
	Let $\delta, \tau \in (0, 1)$ and $1\leq s\leq \frac{n}{n-\delta}$. Assume that $\Omega \subset \mathbb{R}^n$ is {an $s$-distance John domain} with John center $x_J$ and John constant $C_J \geq 1$. Let $B_0 = B(x_J,  d(x_J)/18)$ and let $U \subset \Omega$ be an open set such that $U \cap B_0 = \emptyset$. Then there exist a countable collection of pairwise disjoint balls $\{B(x_i, R_i)\}_{i \in \mathbb{N}}$ such that
	\begin{equation}\label{covering-s-John}
		U \subset  \bigcup_{i \in \mathbb{N}} B(x_i, 5 C_1R_i)
	\end{equation}
	and
	\begin{equation}\label{eq:sJ-box}
		\sum_{i \in \mathbb{N}} R_i^{s(n - \delta)} \leq C_2 (1 - \delta) \mathcal P_{\delta, \tau}(U, \Omega),
	\end{equation}
	where $C_2$ depends only on $n$ and
	$C_1\ge 1$
	depends only on  $\tau$, $s$, $C_J$ and $\diam \Omega$.
\end{prop}
\begin{proof}
Since $\Omega$ is an $s$-distance John domain, it is bounded, and so $0<\operatorname{diam}\Omega<\infty$.
	Let $x\in U$ and let $\gamma: [0, 1]\rightarrow \Omega$ be a
	John curve  in $\Omega$ that joins $x$ with $x_J$ such that $\gamma(0)=x$, $\gamma(1)=x_J$ and
	$$d(\gamma(t),\partial \Omega)> C_J^{-1}|\gamma(t)-x|^{s}$$ 
	for all $0\leq t\leq 1.$
	
	Set
	\begin{equation}\label{tau_0}
		\tau_0:=\min\left\{\frac{\tau}{\diam \Omega}, \tau\right\}<1 \quad\text{and}\quad B_t:=B\left(\gamma(t), \frac{\tau_0 |\gamma(t)-x|^{s}}{18C_J}\right).
	\end{equation}
	Since $U$ is open,
	for sufficiently small $t>0$ we have $B_t\subset U$.
	When $t=1$, then $B_t=B(x_J, \tau_0 C_J^{-1}|x-x_0|^s/18)\subset B(x_J, \tau_0 d(x_J)/18) \subset B_0$ and hence it is disjoint from $U$. Thus if we trace along curve $\gamma,$ there is a ``time" $t$ such that $|U\cap B_t|=|B_t|/2$. This means that there is a point $y$ on the John curve $\gamma$ and a radius $R_y$ with
	\begin{align*}
		&(1)~|U\cap B(y, R_y)|=|B(y,R_y)|/2;\\
		&(2)~R_y\leq \tau_0 d(y)/18;\\
		&(3)~x\in {B\big(y, (36\tau_0^{-1}C_JR_y)^{1/s}\big).}
	\end{align*}
	{Condition
	$(2)$ is clear from \eqref{tau_0} and the definition of an $s$-distance John domain.}
 	For Condition (3), note that
 	$|x-y|=(18\tau_0^{-1} C_J R_y)^{1/s}<  (36\tau_0^{-1} C_J R_y)^{1/s}$.
	
	Let $C_1:=(36\tau_0^{-1}C_J)^{1/s}\geq 1$. Then the family $\big\{B(y, C_1 R_y^{1/s})\big\}_y$ forms a covering
	of $U$.
	According to the
$5$-covering
	theorem (cf. \cite[Theorem 1.24]{EGbook}), we can select a subfamily $\big\{B(y_i, C_1 R_{y_i}^{1/s})\big\}_i$ of pairwise disjoint balls with 
	\begin{equation}\label{cover}
		U\subset \bigcup_{i=1}^{\infty} B\big(y_i, 5 C_1 R_{y_i}^{1/s}\big).
	\end{equation}
	By taking $x_i=y_i$ and $R_i= R_{y_i}^{1/s}$, it is clear that \eqref{covering-s-John} holds with {$C_1=(36\tau_0^{-1}C_J)^{1/s}$ }depending only on $\tau,$ $s,$ $C_J$ and $\diam \Omega$.
	
	To simplify the notation, we write $B(y_i, R_{y_i})$ as
	$B_{y_i}$.
	By condition $(1)$ above, we see that $|U\cap B_{y_i}|=|B_{y_i}|/2$ for every $1\leq i<\infty$. Then it follows from Lemma~\ref{lem:PS-lem} that there is a constant $C'$ depending only on $n$, so that for each $1\leq i<\infty$,
	\begin{align}
		R_{y_i}^{n-\delta}&\leq C' (1-\delta) \int_{U\cap B_{y_i}}\int_{B_{y_i}\setminus U}\frac{dy\,dz}{|z-y|^{n+\delta}}\notag\\
		& =\frac{C' (1-\delta)}{2} \int_{B_{y_i}} \int_{B_{y_i}}  \frac{|\chi_U(z)-\chi_U(y)|}{|z-y|^{n+\delta}} \, dy\, dz.
		\label{estimate-R}
	\end{align}
	
	By using the condition $(2)$ and the choice of $\tau_0$ in \eqref{tau_0}, we obtain that for each $1\leq i<\infty$,
	\begin{equation}\label{bound-R_y}
		R_{y_i}\leq \frac{\tau_0}{18} d(y_i)=\frac{\tau}{18}\min\left\{\frac{d(y_i)}{\diam \Omega}, d(y_i)\right\}\leq \frac{\tau}{18} \min\{1, d(y_i)\}.
	\end{equation}
	For any $z\in B_{y_i}$, since $0<\tau<1$ and $d(z)\geq d(y_i)-R_{y_i}$, we obtain  from \eqref{bound-R_y} that 
	\[18R_{y_i}\leq \tau d(y_i)\leq \tau(d(z)+ R_{y_i})\leq\tau d(z)+ R_{y_i}. \]
	Then $|z-y|<2R_{y_i} <\tau d(z)$ for any $y\in B_{R_{y_i}}$. Therefore, $B_{y_i}\subset B(z, \tau d(z))$ for any $z\in B_{y_i}$. Plugging this into the estimate \eqref{estimate-R},
	we get
	\begin{equation}\label{estimate-A-2}
		R_{y_i}^{n-\delta}\leq\frac{C' (1-\delta)}{2} \int_{B_{y_i}}\int_{|y-z|< \tau d(z)} \frac{|\chi_U(z)-\chi_U(y)|}{|z-y|^{n+\delta}} \, dy\,dz.
	\end{equation} 
	Consequently,
	\begin{align}
		\sum_{i=1}^{\infty} R_i^{s(n-\delta)} &=	\sum_{i=1}^{\infty} \left( R_{y_i}^{1/s}\right)^{s(n-\delta)}
		\leq    \sum_{i=1}^{\infty} \frac{C'(1-\delta) }{2} \int_{B_{y_i}}\int_{|y-z|< \tau d(z)} \frac{|\chi_U(z)-\chi_U(y)|}{|z-y|^{n+\delta}} \, dy\,dz.\label{estimate-A-1}
	\end{align}
	
	It follows from \eqref{bound-R_y} that $R_{y_i}<1$ and $B_{y_i}\subset \Omega$ for each $1\leq i<\infty$.  Since  $\big\{B(y_i,C_1 R_{y_i}^{1/s})\big\}_i$ are pairwise disjoint and $C_1\geq 1$, we know that $\{B_{y_i}\}_i$ are also pairwise disjoint. Hence, we derive from \eqref{estimate-A-1} that
	\begin{align*}
		\sum_{i=1}^{\infty}R_i^{s(n-\delta)}&\leq  \frac{C' }{2}(1-\delta) 
		\int_{\bigcup_{i=1}^{\infty} B_{y_i}}\int_{|y-z|< \tau d(z)} \frac{|\chi_U(z)-\chi_U(y)|}{|z-y|^{n+\delta}} \, dy\,dz\leq \frac{C' }{2} (1-\delta)\mathcal P_{\delta, \tau}(U,\Omega),
	\end{align*}
	which gives \eqref{eq:sJ-box} with $C_2={C' }/{2}$ depending only on $n$. Therefore,	 the proof of Proposition~\ref{prop:sJ-box} is complete.
\end{proof}

	\begin{rem}\label{rem:C_1}
		It can be observed from the proof of the Proposition \ref{prop:sJ-box} that when \(\Omega\) is a 1-John domain, \(C_1\) in equation \eqref{covering-s-John}  is in fact independent of \(\text{diam}\,\Omega\), depending only on \(\tau\) and \(C_J\). The reasoning is as follows: When \(\Omega\) is an $s$-distance John domain, our choice of \(\tau_0\) ensures \(R_{y_i} < 1\), thereby satisfying \(B_{y_i} \subset B\big(y_i, C_1 R_{y_i}^{1/s}\big)\)  and guaranteeing that the balls \(\{B_{y_i}\}_i\) are pairwise disjoint; when \(\Omega\) is a 1-John domain, \(\{B_{y_i}\}_i\) are automatically pairwise disjoint, so we can simply take \(\tau_0 = \tau\).
		It then follows from the definition of $C_1$ that it is independent of \(\text{diam}\,\Omega\).
	\end{rem}

\begin{prop}\label{lem:sJ-WPI-B0}
	Let $\delta, \tau \in (0, 1),$ and $1\leq s\leq \frac{n}{n-\delta}$. Suppose that $\Omega \subset \mathbb{R}^n$ is an $s$-distance John domain with John center $x_J$ and John constant $C_J \geq 1$. Let $B_0 = B(x_J, d(x_J)/18)$. Then there is  a constant $C_3=C_3(n, \tau, s, C_J, \diam \Omega)$
	such that the inequality
	\begin{align}\label{eq:sJ-WPI-B0}
		\int_\Omega |f|\, d\mathcal{H}_\infty^{s(n-\delta)}\leq C_3(1-\delta)\int_\Omega \int_{\left\lvert x-y\right\rvert <\tau d(y)}\frac{\left\lvert f(x)-f(y)\right\rvert }{\left\lvert x-y\right\rvert^{n+\delta} }\,dx\,dy,
	\end{align}
	holds for  every $f\in C(\Omega)$ with $f|_{B_0}=0$.
\end{prop}
\begin{proof}
	By the assumptions on $s$ and $\delta$,  it is clear that 
	\begin{equation}\label{eq:estimate-n-alpha}
		n-1<s(n-\delta)\leq n.
	\end{equation}
	
	We first show that there is a constant $C^\prime=C^{\prime}(n, \tau, s, C_J, \diam\Omega)$
	so that  the following inequality
	\begin{equation}\label{estimate-A}
		\mathcal{H}_\infty^{s(n-\delta)}(U)\leq C^\prime (1-\delta) \int_\Omega \int_{\left\lvert y-z\right\rvert < \tau d(z)}\frac{\left\lvert \chi_{U}(z)-\chi_{U}(y)\right\rvert }{\left\lvert z-y\right\rvert^{n+\delta} }\,dy\,dz
	\end{equation}
	holds for any open set $U\subset \Omega$ with $U\cap B_0=\emptyset$. 
	
	For any open set $U\subset \Omega$ with $U\cap B_0=\emptyset$, it follows from {Proposition~\ref{prop:sJ-box}} 
	that there is a  countable collection of disjoint balls $\{B(x_i,R_i)\}_{i\in\IN}$ such that
	\[
	U \subset \bigcup_{i \in \mathbb{N}} B(x_i, 5 C_1R_i)
	\]
	and
	\[
	\sum_{i \in \mathbb{N}} R_i^{s(n - \delta)} \leq C_2 (1 - \delta) \mathcal P_{\delta, \tau}(U, \Omega).
	\]
	Then it follows from \eqref{eq:estimate-n-alpha} that
	\begin{align*}
		\mathcal{H}_\infty^{s(n-\delta)}(U)&\leq \sum_{i=1}^{\infty}\omega_{s(n-\delta)} (5C_1R_i)^{s(n-\delta)}\\
		&\leq \omega_{s(n-\delta)} (5C_1)^n \sum_{i=1}^{\infty} R_i^{s(n-\delta)}\leq 5^n \omega_{s(n-\delta)} C_1^n C_2(1-\delta) \mathcal P_{\delta, \tau}(U, \Omega)   \\
		&= 5^n \omega_{s(n-\delta)} C_1^n C_2(1-\delta)\int_{\Omega} \int_{\left\lvert y-z\right\rvert <\tau d(z)}\frac{\left\lvert \chi_{U}(z)-\chi_{U}(y)\right\rvert }{\left\lvert z-y\right\rvert^{n+\delta} }\,dy\,dz.
	\end{align*}
	By \eqref{eq:estimate-n-alpha}, it is clear that  $ \omega_{s(n-\delta)}\leq C_\omega$ for some $C_\omega$ depending only on $n$. Therefore, the estimate \eqref{estimate-A} follows by taking $C^\prime=5^n C_\omega C_1^n C_2$. Since $C_1$ and $C_2$ are the constants in Proposition \ref{prop:sJ-box}, $C^\prime$ depends only on $n,$ $s,$ $\tau,$  $C_J$  and $\diam \Omega$.

	Let $\Omega_\lambda=\left\{x\in\Omega:\left\lvert f\right\rvert>\lambda  \right\}$ with $\lambda\geq 0$. By the definition of the Choquet integral,
	\begin{align*}
		\int_\Omega |f|\, d\mathcal{H}_\infty^{s(n-\delta)}
		&=\int_0^\infty \mathcal{H}_{\infty}^{s(n-\delta)}(\Omega_t)\,dt.
	\end{align*}
	Since $f\in C(\Omega)$ with $f|_{B_0}=0$, $\Omega_t$ is an open subset of $\Omega$ with $\Omega_t\cap B_0=\emptyset$ for any $t\geq 0$. Applying the estimate \eqref{estimate-A} to $\Omega_t$ in the above inequality,  we obtain that
	\begin{align}
		\int_\Omega |f|\, d\mathcal{H}_\infty^{s(n-\delta)}&
		\leq C^\prime(1-\delta) \int_{0}^{\infty}\int_{\Omega} \int_{\left\lvert y-z\right\rvert < \tau d(z)}\frac{\left\lvert \chi_{\Omega_t}(z)-\chi_{\Omega_t}(y)\right\rvert }{\left\lvert z-y\right\rvert^{n+\delta} }\,dy\,dz\,dt.\label{coarea}
	\end{align}
	
	For every $z\in \Omega$ and $y\in B(z, \tau d(z))$, we have
	\[\int_{0}^{\infty} \left\lvert \chi_{\Omega_t}(z)-\chi_{\Omega_t}(y)\right\rvert dt= \big\lvert|f(z)|-|f(y)| \big\rvert\leq |f(z)-f(y)|.\]
	Hence, for $y\not=z$, 
	\begin{equation} \label{fubini-pre}
		\int_{0}^{\infty}\frac{\left\lvert \chi_{\Omega_t}(z)-\chi_{\Omega_t}(y)\right\rvert }{\left\lvert z-y\right\rvert^{n+\delta} } \,dt  \leq \frac{|f(z)-f(y)|}{|z-y|^{n+\delta}}.
	\end{equation} 
	Therefore,	by applying the Fubini theorem to the estimate \eqref{coarea}, we obtain that
	\begin{equation*}
		\int_\Omega |f|\, d\mathcal{H}_\infty^{s(n-\delta)}\leq C^\prime(1-\delta) \int_{\Omega} \int_{\left\lvert y-z\right\rvert < \tau d(z)}\frac{\left\lvert f(z)-f(y)\right\rvert }{\left\lvert z-y\right\rvert^{n+\delta} }\,dy\,dz,
	\end{equation*}
	which gives \eqref{eq:sJ-WPI-B0} with $C_3=C^\prime$.
\end{proof}
The main result of this section is the following theorem, which establishes the functional Boxing inequality for continuous functions on $s$-distance John domains.
\begin{thm}\label{thm:sJ-WPI}
	Let $\delta, \tau \in (0, 1),$ and $1\leq s\leq \frac{n}{n-\delta}$. Suppose that $\Omega \subset \mathbb{R}^n$ is an $s$-distance John domain with John center $x_J$ and John constant $C_J \geq 1$.
	{Let $B_0=B(x_J, d(x_J)/18)$.}
	Then there exists  a constant $C_4=C_4(n, \tau, s, C_J, \diam\Omega)$
	such that for every $f\in C(\Omega)$,
	\begin{align}\label{eq:sJ-WPI}
\int_\Omega \left | f-\vint {3B_0}f \right|\, d\mathcal{H}_\infty^{s(n-\delta)}
		\leq C_4(1-\delta)\int_\Omega \int_{\left\lvert x-y\right\rvert <\tau d(y)}\frac{\left\lvert f(x)-f(y)\right\rvert }{\left\lvert x-y\right\rvert^{n+\delta} }\,dx\,dy.
	\end{align}
\end{thm}

\begin{proof}
	Set
	\begin{equation}\label{eq:tau-kappa}
		\kappa=\frac{20\,\diam \Omega}{d(x_J)} \quad  
		{\text{and}} 
		\quad \tau_0^{\prime}=\min\left\{\frac{1}{\kappa}, \tau\right\}. 
	\end{equation}
	Then it is clear that $3B_0\Subset \Omega\Subset \kappa B_0.$ Note that the energy on the right hand of \eqref{eq:sJ-WPI} is nondecreasing with respect to $\tau$. Thus it suffices to show that
	$$\int_\Omega \left\lvert f-\vint{3B_0}f \right\rvert\, d\mathcal{H}_\infty^{s(n-\delta)} \leq C_4(1-\delta)\int_\Omega \int_{\left\lvert x-y\right\rvert <\tau_0^{\prime} d(y)}\frac{\left\lvert f(x)-f(y)\right\rvert }{\left\lvert x-y\right\rvert^{n+\delta} }\,dx\,dy,$$
	where $C_4=C(n, \tau_0^{\prime}, s, C_J, \diam\Omega)$.

	Take a smooth function $\phi$ on $\Omega$ such that
	$$0\leq\phi\leq 1, \quad \phi|_{B_0}=1, \quad \Lip \phi \leq 36/d(x_J),$$
	and  the support of $\phi$ is contained in $2B_0$.  Then we decompose $f-f_{3B_0}$ as
	\begin{equation}\label{eq:decomposition}
		f-f_{3B_0}=\phi(f-f_{3B_0})+(1-\phi)(f-f_{3B_0})=:f_1+f_2.
	\end{equation}
	
	Since $|f_1+f_2|\leq |f_1|+|f_2|$, it follows   that
	\begin{align}\label{eq:norm-decomposition}
		\int_\Omega \left\lvert f-\vint {3B_0}f  \right\rvert\, d\mathcal{H}_\infty^{s(n-\delta)}
		&\leq 2\int_\Omega |f_1|\, d\mathcal{H}_\infty^{s(n-\delta)}+2\int_\Omega |f_2|\, d\mathcal{H}_\infty^{s(n-\delta)}.
	\end{align}

	First we estimate the  term $\int_\Omega |f_1|\, d\mathcal{H}_\infty^{s(n-\delta)}$. By H\"older's inequality for the Choquet integral, we have
	\begin{equation}\label{eq:estimate-f-1-first}
		\int_\Omega |f_1|\, d\mathcal{H}_\infty^{s(n-\delta)}\leq 2\big(\mathcal {H}_\infty^{s(n-\delta)}(\Omega)\big)^{\frac{s-1}{s}}\left(\int_{\Omega} |f_1|^{s}\, d\mathcal{H}_\infty^{s(n-\delta)}\right)^{1/s}.
	\end{equation}	
	Since
	$n,$ $\delta$ and $s$ satisfy \eqref{eq:estimate-n-alpha},
	there exists a constant $C_\omega^{\prime}$ depending only on $n$ and $\diam\Omega$ such that
	\begin{equation}\label{eq:estimate-content-holder}
		\left(\mathcal {H}_\infty^{s(n-\delta)}(\Omega)\right)^{\frac{1}{s}}\leq \left(\mathcal {H}_\infty^{s(n-\delta)}\left(B(x_J, \diam\Omega)\right)\right)^{\frac{1}{s}} =\omega_{s(n-\delta)}^{\frac{1}{s}} \left(\diam\Omega\right)^{n-\delta} \leq C_\omega^{\prime}.
	\end{equation}
	From the decomposition \eqref{eq:decomposition}, we know that $|f_1|\leq |f-f_{3B_0}|$ and the support of $f_1$ is contained in $2B_0$ and hence in $3B_0$. Therefore, it follows from \eqref{eq:estimate-f-1-first} and \eqref{eq:estimate-content-holder} that
	\begin{align}\label{estimate-f1}
		\int_\Omega |f_1|\, d\mathcal{H}_\infty^{s(n-\delta)}&\leq 2(C_\omega^{\prime})^{s-1}\left(\int_{3B_0} |f_1|^{s}\, d\mathcal{H}_\infty^{s(n-\delta)}\right)^{1/s}\notag\\
		&\leq 2(C_\omega^{\prime})^{s-1}\left(\int_{3B_0} |f-f_{3B_0}|^{s}\, d\mathcal{H}_\infty^{s(n-\delta)}\right)^{1/s}.
	\end{align}
	By \cite[Corollary 4.3]{MPW24}, there exists a constant $C'$ depending only on $n$ such that
	\begin{align}\label{eq:estimate-f1}
		\left(\int_{3B_0} \left\lvert f-\vint {3B_0}f\right\rvert^{s}\, d\mathcal{H}_\infty^{s(n-\delta)}\right)^{1/s}\leq C'(1-\delta)\int_{3B_0} \int_{3B_0}\frac{\left\lvert f(x)-f(y)\right\rvert }{\left\lvert x-y\right\rvert^{n+\delta} }\,dx\,dy.
	\end{align}
	Plugging \eqref{eq:estimate-f1} into \eqref{estimate-f1}, we obtain that
	\begin{equation}\label{eq:final-estimate-f1}
		\int_\Omega |f_1|\, d\mathcal{H}_\infty^{s(n-\delta)}\leq 2C'(C_\omega^{\prime})^{s-1}(1-\delta)\int_{3B_0} \int_{3B_0}\frac{\left\lvert f(x)-f(y)\right\rvert }{\left\lvert x-y\right\rvert^{n+\delta} }\,dx\,dy.
	\end{equation}
	
	Let $m$ be the smallest integer satisfying $m\geq 1/\tau_0^{\prime}$.
	For any $x,y \in 3B_0$ and $0\leq k\leq m$, let $Q^k_{x,y}=\frac{k}{m}x+\frac{m-k}{m}y$. Then  it is clear  that 
	\begin{equation}\label{Q-k}
		Q^k_{x,y}\in 3B_0\quad \text{and}\quad |Q^k_{x,y}-Q^{k-1}_{x,y}|=\frac{|y-x|}{m}\leq \frac{d(x_J)}{3m}.
	\end{equation}	
	Since $d(Q^k_{x,y})\geq d(x_J)-d(x_J)/6=5d(x_J)/6$, we have
	\begin{equation}\label{Q-k-1}
		|Q^k_{x,y}-Q^{k-1}_{x,y}|<\frac{d(Q^k_{x,y})}{m}\leq \tau_0^{\prime} d(Q^k_{x,y}).
	\end{equation}
	By the triangle inequality, 
	$$|f(x)-f(y)|\leq \sum_{k=1}^{m} \left\lvert f(Q^{k}_{x,y})-f(Q^{k-1}_{x,y})\right\rvert.$$	
	Therefore, we obtain
	$$\int_{3B_0}\int_{3B_0}\frac{\left\lvert f(x)-f(y)\right\rvert }{\left\lvert x-y\right\rvert^{n+\delta} }\,dx\,dy\leq \sum_{k=1}^{m}\int_{3B_0}\int_{3B_0}\frac{\left\lvert f(Q^{k}_{x,y})-f(Q^{k-1}_{x,y})\right\rvert }{\left\lvert x-y\right\rvert^{n+\delta} }\,dx\,dy.$$
	We change variables: $\tilde{x}=Q^{k}_{x,y}$ and $\tilde{y}=Q^{k-1}_{x,y}$.
	Then the Jacobian matrix is
	$$J = \begin{bmatrix}
		\frac{k}{m} I & \frac{m-k}{m} I \\
		\frac{k-1}{m} I & \frac{m-k+1}{m} I
	\end{bmatrix},$$
	where $I$ denotes the $n\times n$ identity matrix.
	A direct computation gives that the absolutely value of the determinant of $J$ is $1/m^n$. Then it follows from \eqref{Q-k} and \eqref{Q-k-1} that 
	\begin{align}\label{estimat-f1-1}
		\int_{3B_0}\int_{3B_0}\frac{\left\lvert f(x)-f(y)\right\rvert }{\left\lvert x-y\right\rvert^{n+\delta} }\,dx\,dy&\leq  m^{1-\delta} \int_{3B_0} \int_{\left\lvert \tilde{x}-\tilde{y}\right\rvert <\tau d(\tilde{y})}\frac{\left\lvert f(\tilde{x})-f(\tilde{y})\right\rvert }{\left\lvert \tilde{x}-\tilde{y}\right\rvert^{n+\delta} }\,d\tilde{x}\,d\tilde{y}\notag\\
		&\leq \frac{2}{\tau_0^{\prime}} \int_{\Omega} \int_{\left\lvert \tilde{x}-\tilde{y}\right\rvert <\tau d(\tilde{y})}\frac{\left\lvert f(\tilde{x})-f(\tilde{y})\right\rvert }{\left\lvert \tilde{x}-\tilde{y}\right\rvert^{n+\delta} }\,d\tilde{x}\,d\tilde{y},
	\end{align}
	because of the fact $1\leq m<2/\tau_0^{\prime}$ and $\delta\in (0, 1)$.
	Plugging the estimate \eqref{estimat-f1-1} into \eqref{eq:final-estimate-f1}, we obtain 
	\begin{equation}\label{eq:estimate-f-1}
		\int_\Omega |f_1|\, d\mathcal{H}_\infty^{s(n-\delta)}\leq \frac{4C'(C_{\omega}^{\prime})^{s-1}}{\tau_0^{\prime}} (1-\delta) \int_\Omega \int_{\left\lvert x-y\right\rvert <\tau d(y)}\frac{\left\lvert f(x)-f(y)\right\rvert }{\left\lvert x-y\right\rvert^{n+\delta} }\,dx\,dy.
	\end{equation}	
	
	Next we estimate the term  $\int_\Omega |f_2|\, d\mathcal{H}_\infty^{s(n-\delta)}$. From the decomposition \eqref{eq:decomposition}, we have the fact that  $f_2|_{B_0}=0$ and $f_2| _{\Omega \setminus 2B_0}=f-f_{3B_0}.$
	By applying Proposition \ref{lem:sJ-WPI-B0} to $f_2$, we have 
	\begin{align}\label{estimate-f2}
		\int_\Omega |f_2|\, d\mathcal{H}_\infty^{s(n-\delta)}&\leq C_3 (1-\delta) \int_{\Omega}\int_{|x-y|< \tau_0^{\prime} d(y)} \frac{\left\lvert f_2(x)-f_2(y)\right\rvert }{\left\lvert x-y\right\rvert^{n+\delta}}\, dx\,dy\notag\\
		&= C_3 (1-\delta) \int_{\Omega\setminus 3B_0}\int_{|x-y|< \tau_0^{\prime} d(y)} \frac{\left\lvert f_2(x)-f_2(y)\right\rvert }{\left\lvert x-y\right\rvert^{n+\delta}}\, dx\,dy\notag\\
		&\quad\quad+ C_3 (1-\delta) \int_{ 3B_0}\int_{|x-y|< \tau_0^{\prime} d(y)} \frac{\left\lvert f_2(x)-f_2(y)\right\rvert }{\left\lvert x-y\right\rvert^{n+\delta}}\, dx\,dy=:I_1+I_2,
	\end{align}
	where $C_3=C_3(n,\tau_0^{\prime},s,C_J,\diam\Omega)$ is the constant in Proposition \ref{lem:sJ-WPI-B0}. 
	
	To estimate $I_1$, for any $y\in \Omega\setminus 3B_0$, the choice of $\tau_0^{\prime}$ in \eqref{eq:tau-kappa} implies 
	$$\tau_0^{\prime} d(y)\leq \tau_0^{\prime}\,\diam\Omega\leq 
	{\frac{\diam\Omega}{\kappa}}
	=\frac{d(x_J)}{20}<\frac{d(x_J)}{18}.$$
	Hence every $x\in \Omega$ with $|x-y|<\tau_0^{\prime} d(y)$ belongs to $\Omega\setminus 2B_0$. Therefore, by using the fact that $f_2| _{\Omega \setminus 2B_0}=f-f_{3B_0}$, we obtain that
	\begin{align}\label{eq:estimate-I_1}
		I_1&= C_3  (1-\delta) \int_{\Omega\setminus 3B_0}\int_{|x-y|< \tau_0^{\prime} d(y)} \frac{\left\lvert f(x)-f(y)\right\rvert }{\left\lvert x-y\right\rvert^{n+\delta}}\, dx\,dy\notag\\
		&\leq C_3  (1-\delta) \int_{\Omega}\int_{|x-y|< \tau_0^{\prime} d(y)} \frac{\left\lvert f(x)-f(y)\right\rvert }{\left\lvert x-y\right\rvert^{n+\delta}}\, dx\,dy.
	\end{align}
	
	Finally, we estimate $I_2$. By the triangle inequality, we have that
	\begin{align*}
		|f_2(y)-f_2(x)|&\leq |(1-\phi(y))(f(y)-f_{3B_0})-(1-\phi(x))(f(y)-f_{3B_0})|\\
		&\qquad\qquad+|(1-\phi(x))(f(y)-f_{3B_0})-(1-\phi(x))(f(x)-f_{3B_0})|\\
		&=|f(y)-f_{3B_0}||\phi(x)-\phi(y)|+|1-\phi(x)||f(x)-f(y)|.
	\end{align*}
	Since $0\leq \phi(x)\leq 1$ and $\Lip \phi\leq 36/d(x_J)$, we obtain that
	$$	|f_2(y)-f_2(x)| \leq \frac{36|f(y)-f_{3B_0}|}{d(x_J)}|x-y|+|f(x)-f(y)|.$$
	Consequently, by the Fubini theorem,
	\begin{align}\label{estimate-f2-2}
		I_2&\leq \frac{36 C_3(1-\delta)}{d(x_J)}\int_{3B_0}|f(y)-f_{3B_0}|\,dy\int_{|x-y|<\tau_0^{\prime} d(y)} \frac{1}{|x-y|^{n+\delta-1}}\, dx\notag\\
		&\quad\quad\qquad+ C_3(1-\delta) \int_{ 3B_0}
		{\int_{|x-y|< \tau_0' d(y)}} 
		\frac{\left\lvert f(x)-f(y)\right\rvert }{\left\lvert x-y\right\rvert^{n+\delta}}\,  dx\,dy=: I_2^1+I_2^2.
	\end{align}
	
	It is clear that
	\begin{equation}\label{eq:I_2-2}
		I_2^2\leq C_3(1-\delta) \int_{\Omega}
			{\int_{|x-y|< \tau_0' d(y)}} 
		\frac{\left\lvert f(x)-f(y)\right\rvert }{\left\lvert x-y\right\rvert^{n+\delta}}\,  dx\,dy.
	\end{equation}
	
	For the estimate of $I_2^1$, since $d(y)\leq d(x_J)+|y-x_J|\leq 2 d(x_J)$, we have 
	\begin{equation}\label{eq:x-y}
		\int_{|x-y|<\tau_0^{\prime} d(y)} \frac{1}{|x-y|^{n+\delta-1}}\, dx\leq \int_{0}^{2d(x_J)} \frac{n\omega_n}{t^{\delta}}\, dt \leq \frac{2}{1-\delta}n\omega_n d(x_J)^{1-\delta}.
	\end{equation}
	Moreover, by \cite[page 80]{BBM02}, there exists a constant $C''$ depending only on $n$ such that
	\begin{align}\label{estimate-f2-1}
		\int_{3B_0}|f(y)-f_{3B_0}|\,dy\leq C''d(x_J)^{\delta}(1-\delta)\int_{3B_0} \int_{3B_0}\frac{\left\lvert f(x)-f(y)\right\rvert }{\left\lvert x-y\right\rvert^{n+\delta} }\,dx\,dy.
	\end{align}
	By	setting
	\begin{equation}\label{eq:C'_4}
	C'_4=72 n\omega_n C_3C'',
	\end{equation}
	it follows from \eqref{estimat-f1-1}, \eqref{eq:x-y} and \eqref{estimate-f2-1} that
	$$I_2^1\leq C'_4(1-\delta) \int_{3B_0} \int_{3B_0}\frac{\left\lvert f(x)-f(y)\right\rvert }{\left\lvert x-y\right\rvert^{n+\delta} }\,dx\,dy\leq \frac{2C'_4}{\tau_0^{\prime}} (1-\delta) \int_\Omega \int_{\left\lvert x-y\right\rvert <\tau_0^{\prime} d(y)}\frac{\left\lvert f(x)-f(y)\right\rvert }{\left\lvert x-y\right\rvert^{n+\delta} }\,dx\,dy.$$
	Plugging this and \eqref{eq:I_2-2} into \eqref{estimate-f2-2}, we get
	\begin{equation}\label{eq:final-I-2}
		I_2\leq \left(C_3+\frac{2C'_4}{\tau_0^{\prime}}\right)(1-\delta)\int_\Omega \int_{\left\lvert x-y\right\rvert <\tau_0^{\prime} d(y)}\frac{\left\lvert f(x)-f(y)\right\rvert }{\left\lvert x-y\right\rvert^{n+\delta} }\,dx\,dy.
	\end{equation}
	Combining \eqref{eq:norm-decomposition}, \eqref{eq:estimate-f-1}, \eqref{estimate-f2}, \eqref{eq:estimate-I_1} and \eqref{eq:final-I-2}, we have
	\begin{equation}\label{f-estimate-f}
		\int_\Omega \left\lvert f-\vint {3B_0}f\right\rvert\, d\mathcal{H}_\infty^{s(n-\delta)} \leq C_4(1-\delta)\int_\Omega \int_{\left\lvert x-y\right\rvert <\tau_0^{\prime} d(y)}\frac{\left\lvert f(x)-f(y)\right\rvert }{\left\lvert x-y\right\rvert^{n+\delta} }\,dx\,dy,
	\end{equation}
	where $C_4=4C_3+(8C'(C_\omega^{\prime})^{s-1}+4C'_4)/\tau_0^{\prime}$ depends only on  $n$, $\tau_0^{\prime}$, $s$, $C_J,$ $d(x_J)$ and $\diam\Omega$. Therefore, the proof is complete.
\end{proof}

	\begin{rem}\label{rem:C4}
	It is noteworthy that in Theorem \ref{thm:sJ-WPI}, when \(s = 1\), the constant \(C_4\)  in \eqref{eq:sJ-WPI} is independent of \(\diam\Omega\),
	which can be seen as follows.
	We know that \(C_4 = 4C_3 + (8C'(C_\omega')^{s-1} + 4C_4')/\tau_0'\). When \(s = 1\),
	\(\diam\Omega\) and
	\(d(x_J)\)  are comparable; thus, by \eqref{eq:tau-kappa}, \(\tau_0'\) is independent of \(\diam\Omega\). From the proof of Proposition \ref{lem:sJ-WPI-B0}, we have
	\(C_3 = 2 \times 5^n C_\omega C_1^n C_2\).
	According to Remark \ref{rem:C_1}, \(C_1\) is independent of \(\diam\Omega\) when \(s = 1\),
while $C_\omega$ and $C_2$ only depend on $n$.
	Then by \eqref{eq:C'_4} we know that $C_4^{\prime}$ is independent of \(\diam\Omega\).
		Consequently, \(C_4\) is independent of \(\diam\Omega\) when \(s = 1\).
		
		{We can also} understand why \(C_4\) is independent of \(\diam\Omega\) when \(s = 1\) directly from \eqref{eq:sJ-WPI} itself by scaling-invariant.
\end{rem}

\section{\large Fine properties of fractional Sobolev functions}\label{Chapter-4}
In this section, we show that for any bounded domain $\Omega$, smooth functions are dense in $\mathcal {W}^{\delta,1}(\Omega)$ and the exceptional set for
the Lebesgue points of
a function $f\in \mathcal {W}^{\delta,1}(\Omega)$
has $(n-\delta)$-dimensional Hausdorff content zero. 

 A point $x \in \Omega$ is called a Lebesgue point of a function $f \in L^1(\Omega)$ if there exists $a \in \mathbb{R} $ such that
\[
\lim_{r \to 0} \, \vint{B_r(x)} |f(y) - a|\, dy = 0.
\]
Denote by $\mathcal{L}_f$ the set of Lebesgue points of $f$ and we call the set $\Omega\setminus \mathcal{L}_f$ the {\it exceptional set} of $f$. Assigning the value $a$ at the point $x \in \mathcal{L}_f$ yields a function $f^*: \mathcal{L}_f \to \mathbb{R}$ $ (\text{i.e.}, f^*(x) := a)$, called the {\it precise representative} of $f$. Consequently, we have
\[
\lim_{r \to 0} \, \vint{B_r(x)} |f(y) - f^*(x)|\, dy = 0,
\]
and hence
\[
\lim_{r \to 0} \, \vint{B_r(x)} f(y)\, dy = f^*(x).
\]

For any $\delta\in (0,1)$, the $(\delta,1)$-capacity of a compact set $K \subset \mathbb{R}^n$ is defined as:
$$\mathrm{Cap}_{\delta,1}(K) := \inf\left\{ [\varphi]_{W^{\delta,1}(\mathbb{R}^n)} : \varphi \in C_c^\infty(\mathbb{R}^n),\ \varphi\geq 0, \text{ and } \varphi > 1 \text{ on } K \right\}.$$
For a general set $A \subset \mathbb{R}^n$, its capacity  is defined via regularity, for more details see \cite{PS20}.  Capacity is an outer measure commonly used in the study of partial differential equations, often employed to characterize the size
of $\Omega\setminus\mathcal{L}_f$. By \cite[Theorem 2.1]{PS20}, we know that $\mathrm{Cap}_{\delta,1}$ and $\mathcal{H}^{n-\delta}_{\infty}$ are comparable.
Therefore, it is reasonable for us to use $\mathcal{H}^{n-\delta}_{\infty}$ to characterize the size
of \(\Omega\setminus\mathcal{L}_f\) for $f\in \mathcal W^{\delta, 1}(\Omega)$.



It is known that smooth functions are dense in the fractional Sobolev space $W^{\delta, 1}(\Omega)$; see \cite[Theorem 6.65]{Leo23}. We  extend this density result to the improved fractional Sobolev space $\mathcal W^{\delta, 1}(\Omega)$.

\begin{prop}\label{prop:fracSob-dense}
	Let $\Omega \subset \IR^n$ be a bounded domain and $\delta, \tau \in (0,1).$ For every $f\in \mathcal {W}^{\delta,1}(\Omega),$ there exists a sequence $\left\{f_j\right\}$ in $\mathcal {W}^{\delta,1}(\Omega) \cap C^{\infty}(\Omega)$ such that 
	\[
	\left\lVert f-f_j\right\rVert _{\mathcal {W}^{\delta,1}(\Omega)}\rightarrow 0\quad \text{as}\quad j\rightarrow \infty.
	\]
	In other words, $\mathcal {W}^{\delta,1}(\Omega) \cap C^{\infty}(\Omega)$ is dense in $\mathcal {W}^{\delta,1}(\Omega).$
\end{prop}
\begin{proof}
	Let $f\in \mathcal {W}^{\delta,1}(\Omega)$, $\Omega^0=\emptyset$ and 
	\[
	\Omega^{k}:=\left\{x\in \Omega: d(x)>\frac{1}{k}\right\}, \quad k=1,2,\dots.
	\]
	Set
	\[
	U_k=\Omega^{k+1}\setminus\overline{\Omega^{k-1}},\quad k=1,2,\dots.
	\]
	As every $U_k$ is open  and
	\[
	\bigcup_{k=1}^{\infty} U_k=\Omega,
	\]
	there exists a sequence of smooth function $\{\phi_k\}_{k=1}^{\infty}$ such that 
	\begin{equation*}
		\begin{cases}
			\phi_k\in C_{c}^{\infty}(U_k),\ 0\leq \phi_k\leq 1, \quad k=1,2,\dots,\\
			\sum_{k=1}^{\infty}\phi_k=1 ~\text{on}~ \Omega.
		\end{cases}
	\end{equation*}

	\noindent{\bf Claim 1:} For each $k=1,2,\dots,$ we have $f\phi_k\in \mathcal {W}^{\delta,1}(\Omega)$ and $\spt(f\phi_k)\subset U_k$.
	\smallskip
	
	\noindent{\it Proof of Claim 1:}
	By the properties of $\phi_k$, it is clear that $f\phi_k\in L^1(\Omega)$ and $\spt(f\phi_k)\subset U_k$. It remains to show that $[\phi_k f]_{\mathcal {W}^{\delta,1}(\Omega)}<\infty$. By the triangle inequality, a direct computation gives that
	\begin{align*}
		[\phi_k f]_{\mathcal {W}^{\delta,1}(\Omega)}	=&	\int_{\Omega}\int_{|x-y|<\tau d(y)}\frac{\left\lvert \phi_{k}(x)f(x)-\phi_{k}(y)f(y)\right\rvert }{\left\lvert x-y\right\rvert^{n+\delta} }\,dx\,dy\\
		\leq	& \int_{\Omega}\int_{|x-y|<\tau d(y)}\frac{\left\lvert \phi_{k}(x)f(x)-\phi_{k}(x)f(y)\right\rvert }{\left\lvert x-y\right\rvert^{n+\delta} }\,dx\,dy\\
		&\qquad+\int_{\Omega}\int_{|x-y|<\tau d(y)}\frac{\left\lvert \phi_{k}(x)f(y)-\phi_{k}(y)f(y)\right\rvert }{\left\lvert x-y\right\rvert^{n+\delta} }\,dx\,dy=:I_1+I_2.
	\end{align*}
	Since $0\leq \phi\leq 1$, we have
	$$I_1\leq \int_{\Omega}\int_{|x-y|<\tau d(y)}\frac{\left\lvert f(x)-f(y)\right\rvert }{\left\lvert x-y\right\rvert^{n+\delta} }\,dx\,dy=[f]_{\mathcal {W}^{\delta,1}(\Omega)}<\infty.$$
	Since $\phi_k\in C_{c}^{\infty}(U_k)$, it is a Lipschitz function and hence there exists a constant $M$ 
	such that
	$$|\phi_k(x)-\phi_k(y)|\leq M|x-y| \quad \text{for any } x, y\in \Omega.$$
	Consequently, by the Fubini theorem,
	$$I_2\leq \int_\Omega |f(y)|\,dy\int_{|x-y|<\tau d(y)} \frac{M}{|x-y|^{n-1+\delta}}\, dx.$$
	Using the fact that $\tau d(y)<\diam\Omega$ for any $y\in\Omega$, it is clear that
	$$\int_{|x-y|<\tau d(y)} \frac{M}{|x-y|^{n-1+\delta}}\, dx\leq \int_0^{\diam\Omega} \frac{Mn\omega_{n}}{r^{\delta}}\, dr=\frac{Mn\omega_{n}}{1-\delta}(\diam\Omega)^{1-\delta}<\infty.$$
	Together with $f\in L^1(\Omega)$, this shows that $I_2<\infty$. Therefore, we obtain that  $[\phi_k f]_{\mathcal {W}^{\delta,1}(\Omega)}<\infty$ and Claim $1$ holds.
	
	\smallskip

	Let $\eta _{\epsilon}$ be a standard mollifier. For any function $g\in L_{\rm loc}^1(\Omega)$, define 
	\[
	g^{\epsilon}(x):=\eta _{\epsilon}\ast g:=\int_{\Omega}\eta _{\epsilon}(x-y)g(y)\,dy=\int_{B(0,1)}\eta(y)g(x-\epsilon y)\,dy, \quad x\in\Omega_\epsilon,
	\]
	where $\Omega_{\epsilon}=\{x\in \Omega: d(x)>\epsilon\}$. For more details and properties of mollifiers, we refer to \cite[Section 4.2]{EGbook}. Moreover, for any function $g\in L_{\rm loc}^1(\mathbb R^n)$ and any $h\in \mathbb R^n$, let 
	$$\Delta _h g(x)=g(x)-g(x-h).$$
	By the Fubini theorem, a direct computation gives that
	\begin{align}\label{eq:rewrite-norm}
		[g]_{\mathcal {W}^{\delta,1}(\Omega)}&=\int_{\Omega}\int_{|x-y|<\tau d(x)}\frac{|g(x)-g(y)|}{|x-y|^{n+\delta}}\, dy\,dx=\int_{\Omega}\int_{|h|<\tau d(x)}\frac{|\Delta_h g(x)|}{|h|^{n+\delta}}\, dh\,dx\notag\\
		&=\int_{\mathbb R^n}\int_{\Omega_{|h|/\tau}}\frac{|\Delta_h g(x)|}{|h|^{n+\delta}}\,dx\,dh.
	\end{align}
	
	\smallskip
	\noindent{\bf Claim 2:}  For each $k=1,2,\dots,$  we have
	$$
	\lim_{\epsilon\rightarrow 0^+}\left\lVert f\phi_k-(f\phi_k)^\epsilon\right\rVert _{\mathcal {W}^{\delta,1}(\Omega)}=0.
	$$
	
	\noindent{\it Proof of Claim 2:} 
	To simplify the notation, let
	\[
	f_{k}:=f\phi_k\quad\text{and}\quad 	f^\epsilon_{k}:=(f\phi_k)^\epsilon.
	\]
	Since $\spt(f_k)\subset U_{k}$,
	we can extend $f_k$ to $\mathbb R^n$ by using zero extension and then $\|f_k\|_{L^1(\mathbb R^n)}=\|f_k\|_{L^1(\Omega)}<\infty$. Moreover, it follows from the properties of mollifiers (cf. \cite[Theorem 4.1]{EGbook}) that $\spt(f^\epsilon_k)\subset U_{k}$ for sufficiently small $\epsilon>0$ and 
	\begin{align}\label{eq:claim2-l1-estimate}
		\lim_{\epsilon\rightarrow 0^+}	\left\lVert f_k-f^\epsilon_k\right\rVert _{L^1{(\IR^{n})}}=\lim_{\epsilon\rightarrow 0^+}	\left\lVert f_k-f^\epsilon_k\right\rVert _{L^1{(\Omega)}}= 0.
	\end{align} 
	{Since $\Delta_h (f_k-f_k^\epsilon)=\Delta_h f_k-\Delta_h f^\epsilon_k$, }we obtain from \eqref{eq:rewrite-norm} that
	\begin{align*}
		[f_k-f_k^\epsilon]_{\mathcal {W}^{\delta,1}(\Omega)}&=\int_{\IR^n}\int_{\Omega_{|h|/\tau}}\frac{\left\lvert \Delta _h f_k(x) -\Delta_h f_k^{\epsilon}(x)\right\rvert }{|h|^{n+\delta}}\,dx \,dh.
	\end{align*}
	Let 
	$$g_\epsilon(h)=\int_{\Omega_{|h|/\tau}}\frac{\left\lvert \Delta _h f_k(x) -\Delta_h f_k^{\epsilon}(x)\right\rvert }{|h|^{n+\delta}}\,dx.$$
	By \eqref{eq:claim2-l1-estimate}, for every $h\in \IR^n\setminus \{0\}$, we have
	\begin{align*}
		|g_\epsilon(h)|\leq \frac{1}{|h|^{n+\delta}} \int_{\IR^n}{\left\lvert \Delta_h f_k(x) -\Delta_h f_k^\epsilon(x)\right\rvert }\,dx \leq \frac{2}{|h|^{n+\delta}} \|f_k-f_k^\epsilon\|_{L^1(\mathbb R^n)}\rightarrow 0\quad \text{as } \epsilon\rightarrow 0.
	\end{align*}
	Therefore, Claim 2 follows by the Lebesgue dominated convergence theorem, once we find a control function $g\in L^1(\mathbb R^n)$ such that for sufficiently small $\epsilon>0$,
	\begin{equation}\label{control}
		|g_\epsilon (h)|\leq g(h)\quad \text{for a.e. } h\in \mathbb R^n.
	\end{equation}
	
	We claim that the function 
	\begin{align*}
		g(h):=\frac{2}{|h|^{n+\delta}}\int_{\Omega}\left\lvert \Delta_h f_k(x) \right\rvert \chi_{\{|h|< \tau d(x)+\tau/(3k+3) \}}(h)\,dx
	\end{align*}
	is the one we need. 
	First, we verify that $g$ satisfies \eqref{control}.
	Take $0<\epsilon<\frac{1}{3(k+1)}$ small enough such that $\spt(f_k^\epsilon)\subset U_k$.
	By the definition of $f_k^\epsilon$, we have
	\begin{align}\label{eq:g-1}
		\int_{\Omega_{|h|/\tau}}\left\lvert \Delta_h f_k^\epsilon(x)\right\rvert \,dx&=\int_{\Omega_{|h|/\tau}}\left\lvert f_k^\epsilon(x)-f_k^\epsilon(x-h)\right\rvert \,dx\notag\\
		&\leq\int_{\Omega_{|h|/\tau}} \int_{B(0,1)}\eta(y)\left\lvert f_k(x-\epsilon y)-f_k(x-h-\epsilon y)\right\rvert \,dy \,dx\notag\\
		&	=\int_{B(0, 1)} \eta(y)\, dy\int_{\Omega_{|h|/\tau}}\left\lvert f_k(x-\epsilon y)-f_k(x-h-\epsilon y)\right\rvert\, dx.
	\end{align}
	
	If $x\in \Omega_{|h|/\tau}\setminus\Omega^{3(k+1)}$, then $d(x)\leq \frac{1}{3(k+1)}$. Both $x-\epsilon y$ and $x-h-\epsilon y$ belong to $B(x, \tau d(x)+\epsilon)$ whenever $y\in B(0, 1)$. However, because $0<\epsilon<\frac{1}{3(k+1)}$, for any $z\in B(x, \tau d(x)+\epsilon)$
	we have  $d(z)\leq d(x)+\tau d(x)+\epsilon <2d(x)+\epsilon<\frac{1}{k+1}$,
	and so $B(x, \tau d(x)+\epsilon)\cap U_k=\emptyset$. It follows from $\spt(f_k)\in U_k$ that \eqref{eq:g-1} can be rewritten as
	\begin{equation}\label{eq:change-x}
		\int_{\Omega_{|h|/\tau}}\left\lvert \Delta_h f_k^\epsilon(x)\right\rvert \,dx\leq \int_{B(0, 1)} \eta(y)\, dy\int_{\Omega_{|h|/\tau}\cap\Omega^{3(k+1)}}\left\lvert f_k(x-\epsilon y)-f_k(x-h-\epsilon y)\right\rvert\, dx.
	\end{equation}
	We change the variable by taking  $\tilde x=x-\epsilon y$. Then $d\tilde x=dx$. If $x\in \Omega_{|h|/\tau}\cap\Omega^{3(k+1)}$, we have 
	$$d(x)>\max \left\{\frac{|h|}{\tau}, \frac{1}{3(k+1)}\right\},$$
	and hence
	$$d(\tilde x)>d(x)-\epsilon>d(x)-\frac{1}{3(k+1)}> \max\left\{0, \frac{|h|}{\tau}- \frac{1}{3(k+1)}\right\}.$$
	{Therefore, $\tilde x$ belongs to $\Omega$ and  satisfies $|h|<\tau d(\tilde x)+\frac{\tau}{3(k+1)}$. It follows from \eqref{eq:change-x} and the fact $\int_{B(0, 1)}\eta(y)\,dy=1$ that
		\begin{equation}\label{eq:upper-delta}
			\int_{\Omega_{|h|/\tau}}\left\lvert \Delta_h f_k^\epsilon(x)\right\rvert \,dx\leq \int_{\Omega}\left\lvert \Delta_h f_k(\tilde x) \right\rvert \chi_{\left\{|h|< \tau d(\tilde x)+\tau/(3k+3) \right\}}(h)\,d\tilde x =\frac{|h|^{n+\delta}}{2} g(h).
		\end{equation}
	}
	{
		For every $h\in \mathbb R^n\setminus\{0\}$, by the triangle inequality and \eqref{eq:upper-delta}, 
		\begin{align*}
			g_\epsilon(h) &\leq \frac{1}{|h|^{n+\delta}}\int_{\Omega_{|h|/\tau}}\left\lvert \Delta _h f_k(x)\right\rvert\, dx+ \frac{1}{|h|^{n+\delta}}\int_{\Omega_{|h|/\tau}}\left\lvert\Delta_h f_k^{\epsilon}(x)\right\rvert \,dx\\
			&\leq 	\frac{1}{|h|^{n+\delta}}\int_{\Omega}\left\lvert \Delta _h f_k(x)\right\rvert\chi_{\{|h|<\tau d(x)\}}(h)\, dx+\frac{g(h)}{2}\leq g(h),
		\end{align*}
		which gives \eqref{control} as desired.}
	
	It remains to check that $g\in L^{1}(\mathbb R^n)$. Note that by the Fubini theorem,
	\begin{align*}
		\int_{\mathbb R^n} |g(h)|\, dh&=2\int_{\mathbb R^n}\int_{\Omega} \frac{\left\lvert \Delta_h f_k(x) \right\rvert}{|h|^{n+\delta}} \chi_{\{|h|< \tau d(x)+\tau/(3k+3) \}}(h)\,dx\,dh\\
		&=2\int_{\mathbb R^n} \int_{\Omega}\frac{\left\lvert \Delta_h f_k(x) \right\rvert}{|h|^{n+\delta}} \chi_{\{|h|< \tau d(x) \}}(h)\,dx\,dh\\
		&\qquad\qquad+2\int_{\mathbb R^n} \int_{\Omega}\frac{\left\lvert \Delta_h f_k(x) \right\rvert}{|h|^{n+\delta}} \chi_{\{\tau d(x)\leq |h|< \tau d(x)+\tau/(3k+3) \}}(h)\,dx\,dh=:T_1+T_2.
	\end{align*}
	It follows from Claim $1$ and \eqref{eq:rewrite-norm} that 
	$$T_1=2\int_{\mathbb R^n} \int_{\Omega_{|h|/\tau}}\frac{\left\lvert \Delta_h f_k(x) \right\rvert}{|h|^{n+\delta}} \,dx\,dh=2[f_k]_{\mathcal {W}^{\delta,1}(\Omega)}<\infty.$$
	For the estimate of $T_2$, note that for every $x\in\Omega\setminus\Omega^{3(k+1)}$ and every $|h|<\tau d(x)+\tau/(3k+3)$, $d(x-h)<d(x)+|h|<2d(x)+1/(3k+3)\leq 1/(k+1)$ and hence neither $x$ nor $x-h$ belongs to $U_k$. Therefore, it follows from $\spt(f_k)\subset U_k$ that
	\begin{align*}
		T_2&= 2\int_{\mathbb R^n} \int_{\Omega^{3(k+1)}}\frac{\left\lvert \Delta_h f_k(x) \right\rvert}{|h|^{n+\delta}} \chi_{\{\tau d(x)\leq |h|< \tau d(x)+\tau/(3k+3) \}}(h)\,dx\,dh\\
		&	\leq 2 \int_{\mathbb R^n} \int_{\Omega^{3(k+1)}}\left(\frac{3k+3}{\tau}\right)^{n+\delta} (|f_k(x)|+|f_k(x-h)|)\chi_{\{\tau/(3k+3)\leq |h|\leq \tau\diam\Omega+\tau/(3k+3)\}}(h)\, dx\,dh\\
		&	\leq 4 \left(\frac{3k+3}{\tau}\right)^{n+\delta} \omega_n \left(\tau\diam\Omega+\frac{\tau}{3k+3}\right)^n \|f_k\|_{L^1(\mathbb R^n)}<\infty.
	\end{align*}
	Because both $T_1$ and $T_2$ are finite, we get that $g\in L^1(\mathbb R^n)$.
	Therefore, we have completed the proof of Claim 2.
	
	\smallskip
	
	For  any $j\in \mathbb N$, it follows from Claim 2 that for each $k=1, 2, \dots,$ there exists an $\epsilon_k>0$ such that
	$$\left\|f_k^{\epsilon_k} -f_k\right\| _{\mathcal {W}^{\delta,1}(\Omega)} <\frac{1}{2^k j}.$$
	Define
	\[
	f_j=\sum_{k=1}^{\infty}f_k^{\epsilon_k}.
	\]
	In the neighborhood of each point $x\in \Omega$, by the definition of $U_k$, there are only finitely many nonzero terms in above sum. Since $f_k^{\epsilon_k}\in C^{\infty}(\Omega)$ for every $k=1, 2, \dots,$ we know that $f_j\in C^{\infty}(\Omega)$.
	Moreover,
	\[
	\left\lVert f-f_j\right\rVert _{\mathcal {W}^{\delta,1}(\Omega)}\leq \sum_{k=1}^{\infty} \left\|f_k^{\epsilon_k} -f_k\right\| _{\mathcal {W}^{\delta,1}(\Omega)} <\frac{1}{j}.
	\]
	Consequently, $f_j\in \mathcal {W}^{\delta,1}(\Omega)$ and $f_j\rightarrow f$ in $\mathcal {W}^{\delta,1}(\Omega)$ as $j\rightarrow \infty$. Therefore, the proof is completed.
\end{proof}

{
	\begin{rem}\label{rem-dense}
		If $x\in\Omega$ is a Lebesgue point of $f\in W^{\delta,1},$ then the sequence of functions $\{f_{j}\}$ in the proof of Proposition \ref{prop:fracSob-dense} satisfies
		$$\lim_{j\rightarrow \infty}f_{j}(x)=f^*(x);$$
		this can be seen as follows. In the above proof, for any \(x \in \Omega\) there exists a ball \(B(x, r_x)\)
		and $k_x\in \mathbb{N}$ such that
	\(f = \sum_{k=1}^{k_x} f\phi_k\).
		Then,
		{for sufficiently large $j$,
		each $f_j$ is constructed in a smaller ball centered at $x$}
	by mollifying separately each \(f\phi_k\),
			where \(k = 1, \dots, k_x\) and summing over these. According to \cite[Proposition 8.4]{B-P16},
		if $x$ is a Lebesgue point of $f\phi_k$, then $\lim_{k_n\rightarrow \infty}(f\phi_k)^{1/k_n}(x)=f\phi_k(x).$ Thus, it suffices to show that if $x$ is a Lebesgue point of $f$, then $x$ is also a Lebesgue point of each \(f\phi_k\) (for \(k = k_1, \dots, k_x\)). We now verify this assertion. By the smoothness of \(\phi_k\), we have
		\begin{align*}
			\lim_{r\rightarrow 0}\,\vint{B(x,r)}|f(z)\phi(z)-f^*(x)\phi(x)|\,dz &\leq\lim_{r\rightarrow 0}\vint{B(x,r)}|f(z)\phi(z)-f^*(x)\phi(z)|\,dz\\
			&\qquad+\lim_{r\rightarrow 0}\,\vint{B(x,r)}|f^*(x)\phi(z)-f^*(x)\phi(x)|\,dz\\
			&\leq\lim_{r\rightarrow 0}\,\vint{B(x,r)}|\phi(z)||f(z)-f^*(x)|\,dz\\
			&\qquad+\lim_{r\rightarrow 0}\, \vint{B(x,r)}|f^*(x)||\phi(z)-\phi(x)|\,dz =0.
		\end{align*}
	\end{rem}
}

In \cite[Proposition 3.1]{PS20}, Ponce--Spector proved that for any \(f \in W^{\delta,1}(\mathbb{R}^n)\), the \((n-\delta)\)-dimensional Hausdorff content of its exceptional set is zero. The following theorem extends their result to functions in $ \mathcal {W}^{\delta,1}(\Omega)$.
\begin{thm}\label{thm:fine property}
	Let $\delta, \tau\in (0, 1)$, $\beta\geq n-\delta$ and $\Omega\subset \IR^n$ be a  domain. Then for every $f\in \mathcal {W}^{\delta,1}(\Omega)$, the exceptional set of $f$ is a $\beta$-dimensional Hausdorff content zero set, that is
	\[
	\mathcal{H}_{\infty} ^{\beta}(\Omega\setminus \mathcal{L}_f)=0.
	\]
\end{thm}
\begin{proof}
	Let $f\in \mathcal {W}^{\delta,1}(\Omega)$. Given a point $x\in \Omega$, let $B_x=B(x, r_x)$ with $r_x=\tau d(x)/3$. For any $y\in B_x$, since $\tau\in (0, 1)$, we have 
	$$\tau d(y)\geq \tau (d(x)-r_x)>\frac{2\tau}{3} d(x)=2r_x.$$
	Consequently, $B_x\subset B(y, \tau d(y))$ for any $y\in B_x$. Then it is clear that
	$$\int_{B_x}\int_{B_x}\frac{|f(z)-f(y)|}{|z-y|^{n+\delta}}\, dz\,dy\leq \int_{B_x}\int_{|z-y|<\tau d(y)} \frac{|f(z)-f(y)|}{|z-y|^{n+\delta}}\, dz\,dy\leq [f]_{\mathcal {W}^{\delta,1}(\Omega)}<\infty.$$
	
	By an extension result for the fractional Sobolev space (cf. \cite[Theorem 1.1]{Zhou15}), there exists an extension $\hat f$ of $f|_{B_x}$ such that $\hat f\in W^{\delta, 1}(\mathbb R^n)$. Let 
	$$\mathcal {L}_{\hat f}:=\{y\in \mathbb R^n: y \text{ is a Lebesgue point of } \hat f\}.$$
	It is clear that $\mathcal L_f \cap B_x=\mathcal L_{\hat f}\cap B_x$. Therefore, it follows from \cite[Proposition 3.1]{PS20} that
	$$\mathcal H^{n-\delta}_\infty(B_x\setminus \mathcal L_{f})=\mathcal H^{n-\delta}_\infty(B_x\setminus \mathcal L_{\hat f})\leq \mathcal H^{n-\delta}_\infty(\mathbb R^n\setminus \mathcal L_{\hat f})=0.$$
	Let $\mathcal A$ be the collection of all rational points in $\Omega$. Then $\{B_x\}_{x\in \mathcal A}$ forms a countable covering of $\Omega$. Hence,
	$$\mathcal H^{n-\delta}_{\infty}(\Omega\setminus \mathcal L_f)\leq \sum_{x\in \mathcal A} \mathcal H^{n-\delta}_{\infty}(B_x\setminus \mathcal L_f)=0.$$
For any \(\beta\ge 0\), the \(\beta\)-dimensional Hausdorff content \(\mathcal{H}_\infty^\beta\) and the \(\beta\)-dimensional Hausdorff measure \(\mathcal{H}^\beta\) have the same null sets.
	Consequently, \(\mathcal{H}_\infty^{n-\delta}(\Omega\setminus \mathcal L_f) =0= \mathcal{H}^{n-\delta}(\Omega\setminus \mathcal L_f)\) and thus
	 \(\mathcal{H}^\beta(\Omega\setminus \mathcal L_f) = 0 = \mathcal{H}_\infty^\beta(\Omega\setminus \mathcal L_f)\) for all $\beta\geq n-\delta$, which completes the proof of the theorem.
\end{proof}

\section{\large Functional and Geometric Boxing inequalities}\label{sec-5}

In this section, we
present the functional Boxing inequality for \(L^1(\Omega)\) functions (Theorem \ref{functional-boxing}). We then derive the geometric Boxing inequality (Theorem \ref{geometric-boxing}) from the functional Boxing inequality. Finally, we  prove that the geometric and functional Boxing inequalities are equivalent (Proposition~\ref{prop-equivalence}). 

{First we state two lemmas that are still limited to continuous functions.}

\begin{lem}\label{lem:1J-SPI}
	Let $\delta, \tau\in (0, 1)$ and $1\leq s\leq \frac{n}{n-\delta}$. Assume that $\Omega\subset \IR^n$ is an $s$-distance John domain with John constant $C_J\geq 1$. Then there is  a constant $C_5=C_5(n, \tau, s, C_J, \diam \Omega)$
	such that for every $f\in C(\Omega)$,
	$$\left(\int_{\Omega}\left\lvert f-\vint{3B_0} f\right\rvert^{\frac{n}{s(n-\delta)}}\,dx \right)^{\frac{s(n-\delta)}{n}} \leq C_5(1-\delta)\int_\Omega \int_{\left\lvert x-y\right\rvert <\tau d(y)}\frac{\left\lvert f(x)-f(y)\right\rvert }{\left\lvert x-y\right\rvert^{n+\delta} }\,dx\,dy.$$
\end{lem}
\begin{proof}
	Cavalieri's principle and Lemma \ref{lem:q-integral} imply
	\begin{align*}
		\left(\int_{\Omega}\left\lvert f-\vint{3B_0} f\right\rvert^{\frac{n}{s(n-\delta)}}\,dx \right)^{\frac{s(n-\delta)}{n}} 
		\leq 2\int_{0}^{\infty}|\left\{x\in \Omega:|f(x)-f_{3B_0}|>t\right\} |^{\frac{s(n-\delta)}{n}}\,dt.
	\end{align*}
	To estimate the integrand on the right-hand side, we note that for every open set $U$, we have
	\begin{equation}\label{L-less-H}
		|U|^{\frac{s(n-\delta)}{n}}\leq \mathcal{H}_{\infty} ^{s(n-\delta)}(U),
	\end{equation}
	which follows from a covering argument and the concavity of the function $t\in [0,\infty)\rightarrow t^{\frac{s(n-\delta)}{n}}.$
	
	Hence, by Theorem \ref{thm:sJ-WPI}, we have
	\[
	\left(\int_{\Omega}\left\lvert f-\vint{3B_0} f\right\rvert^{\frac{n}{s(n-\delta)}}\,dx \right)^{\frac{s(n-\delta)}{n}}  \leq 2C_4(1-\delta)\int_\Omega \int_{\left\lvert x-y\right\rvert <\tau d(y)}\frac{\left\lvert f(x)-f(y)\right\rvert }{\left\lvert x-y\right\rvert^{n+\delta} }\,dx\,dy.
	\]
	We thus complete the proof by setting \(C_5 = 2C_4\).
\end{proof}
\begin{lem}\label{lem:1J-Fbox}
	Let $\delta, \tau\in (0, 1)$ and $1\leq s\leq \frac{n}{n-\delta}$. Assume that $\Omega\subset \IR^n$ is an $s$-distance John domain with John constant $C_J\geq 1$. Then there is  a constant
	$C_6=C_6(n, \tau, s, C_J, \diam \Omega)$
	such that for every $f\in C(\Omega)\cap L^1(\Omega)$,
	\begin{align*}
		\int_{\Omega}\left\lvert f-\vint{\Omega} f\right\rvert \,d\mathcal{H}_{\infty} ^{s(n-\delta)} \leq C_6(1-\delta)\int_\Omega \int_{\left\lvert x-y\right\rvert <\tau d(y)}\frac{\left\lvert f(x)-f(y)\right\rvert }{\left\lvert x-y\right\rvert^{n+\delta} }\,dx\,dy.
	\end{align*}
\end{lem}
\begin{proof}
	Since $|f-f_{\Omega}|\leq |f-f_{3B_0}|+|f_{\Omega}-f_{3B_0}|$, using the
	monotonicity and sublinearity properties of the Choquet integral and then using H\"older's inequality, we get
	\begin{align}
		\int_{\Omega}\left\lvert f-\vint{\Omega} f\right\rvert \,d\mathcal{H}_{\infty}^{s(n-\delta)} &\leq 2\int_{\Omega}\left\lvert \vint{\Omega} f-\vint{3B_0}f\right\rvert \,d\mathcal{H}_{\infty} ^{s(n-\delta)}+2\int_{\Omega}\left\lvert f-\vint{3B_0} f\right\rvert \,d\mathcal{H}_{\infty} ^{s(n-\delta)}\notag\\
		&\leq {2\mathcal{H}_{\infty} ^{s(n-\delta)}(\Omega)}\vint{\Omega}\left\lvert f-\vint{3B_0} f\right\rvert\, dx+2\int_{\Omega}\left\lvert f-\vint{3B_0} f\right\rvert \,d\mathcal{H}_{\infty} ^{s(n-\delta)}\notag\\
		&\leq \frac {2\mathcal{H}_{\infty} ^{s(n-\delta)}(\Omega)}{|\Omega|^{(s(n-\delta))/n}} \left(\int_{\Omega}\left\lvert f-\vint{3B_0} f\right\rvert^{\frac{n}{s(n-\delta)}}\,dx \right)^{\frac{s(n-\delta)}{n}}\notag\\
	    &\qquad\qquad\qquad\qquad\qquad\qquad+2\int_{\Omega}\left\lvert f-\vint{3B_0} f\right\rvert \,d\mathcal{H}_{\infty}^{s(n-\delta)}.\label{eq:f-f-omega}
	\end{align}
	Let $x_J$ be the John center. Then it is clear that $\Omega \subset B(x_J, \diam\Omega).$ Moreover, there exists a point $y\in \Omega$ with $|y-x_J|\geq \diam\Omega/3$. Let $\gamma$ be a John curve connecting $x_J$ and $y$. By the property \eqref{def-John} of John curves, we know that 
	$$d(x_J)\geq C_J^{-1} |y-x_J|^s\geq \frac{(\diam\Omega )^s}{3^sC_J}.$$
	Consequently,
	$$B\left(x_J, \frac{(\diam\Omega) ^s}{3^sC_J}\right)\subset \Omega\subset B(x_J, \diam\Omega).$$
	Therefore, 
	$$\mathcal{H}_{\infty} ^{s(n-\delta)}(\Omega)\leq \omega_{s(n-\delta)} (\diam\Omega)^{s(n-\delta)}\quad \text{and} \quad |\Omega|^{(s(n-\delta))/n}\geq \left(\frac{\omega_n^{1/n}}{3^sC_J}\right)^{{s(n-\delta)}} (\diam\Omega)^{s^2(n-\delta)}. $$
	By using the
	facts that $n-1<s(n-\delta)<n$ and $1\le s<2$,
	we obtain that
	$$\frac {2\mathcal{H}_{\infty} ^{s(n-\delta)}(\Omega)}{|\Omega|^{\frac{s(n-\delta)}{n}}}
	\leq C(n, C_J){(\diam\Omega)^{s(1-s)(n-\delta)}.}$$
	Hence, by applying Lemma \ref{lem:1J-SPI} and Theorem \ref{thm:sJ-WPI} to \eqref{eq:f-f-omega}, we have 
	\[
	\int_{\Omega}\left\lvert f-\vint{\Omega} f\right\rvert \,d\mathcal{H}_{\infty} ^{s(n-\delta)} \leq C_6(1-\delta)\int_\Omega \int_{\left\lvert x-y\right\rvert <\tau d(y)}\frac{\left\lvert f(x)-f(y)\right\rvert }{\left\lvert x-y\right\rvert^{n+\delta} }\,dx\,dy,
	\]
	where the constant $C_6=\big(C(n,C_J) {(\diam\Omega)^{s(1-s)(n-\delta)}}\big)C_5+C_4$. Thus, the proof is finished.  
\end{proof}

With the above preparations, we are now ready to prove Theorem \ref{functional-boxing}.
\begin{proof}[Proof of Theorem \ref{functional-boxing}]
	We may assume that $f\in \mathcal {W}^{\delta,1}(\Omega)$, since if $f\notin \mathcal {W}^{\delta,1}(\Omega)$, the inequality obviously holds. Relying on Proposition \ref{prop:fracSob-dense} and Remark \ref{rem-dense}, we know that there exist a sequence  $\{f_k\}_{k=1}^{\infty}\subset C^{\infty}(\Omega)\cap \mathcal {W}^{\delta,1}(\Omega)$ such that  
	\begin{equation}\label{eq:approx}
		\left\lVert f-f_k\right\rVert _{\mathcal {W}^{\delta,1}(\Omega)}\rightarrow 0\quad \text{as}~k\rightarrow \infty,
	\end{equation}
	and for any $x\in \mathcal{L}_f$,
\begin{equation}\label{pointwise}
	\lim_{k\rightarrow \infty}f_k(x)=f^*(x).
\end{equation}
Since $(f_k)_{\Omega}\rightarrow f_\Omega$ as $k\rightarrow \infty$, 	by a substitution of $f-f_\Omega$ (resp. $f_{k} - (f_{k})_\Omega$) for $f$ (resp. $f_k$), we may assume that $f_\Omega=0$ and $(f_k)_\Omega = 0$ for every $k$.

The rest of the proof relies on the strong subadditivity of the
dyadic Hausdorff content  $\widehat{\mathcal{H}}_{\infty}^{s(n-\delta)}$.
According to \eqref{strong subadditive},
$\widehat{\mathcal{H}}_{\infty}^{s(n-\delta)}$ satisfies strong subadditivity and is comparable to  $\mathcal{H}_{\infty} ^{s(n-\delta)}$ with constants depending only on $n$.
Therefore, to complete the proof, it is sufficient to show that there is a constant $C=C(n, \tau, s, C_J, \diam \Omega)$
such that 
	\begin{equation}\label{eq:f^*}
		\int_{\Omega} |f^*| \,d\widehat{\mathcal{H}}_{\infty}^{s(n-\delta)}\leq C(1-\delta)\int_\Omega \int_{\left\lvert x-y\right\rvert <\tau d(y)}\frac{\left\lvert f(x)-f(y)\right\rvert }{\left\lvert x-y\right\rvert^{n+\delta} }\,dx\,dy.
	\end{equation}
	
	By using Lemma \ref{lem:1J-Fbox} and the comparability of $\widehat{\mathcal{H}}_{\infty}^{s(n-\delta)}$ and $\mathcal{H}_{\infty} ^{s(n-\delta)}$, there is  a constant $C_6=C(n, \tau, s, C_J, \diam \Omega)$
	such that 
	\begin{align}\label{smooth}
		\int_{\Omega}|f_k| \,d\widehat{\mathcal{H}}_{\infty}^{s(n-\delta)} \leq C_6(1-\delta)\int_\Omega \int_{\left\lvert x-y\right\rvert <\tau d(y)}\frac{\left\lvert f_k(x)-f_k(y)\right\rvert }{\left\lvert x-y\right\rvert^{n+\delta} }\,dx\,dy
	\end{align}
	holds for every $k$. 
	Through the pointwise convergence \eqref{pointwise}, we obtain that for every $t>0$,
	\[
	\left\{x\in \Omega: |f^*(x)|>t\right\}\subset \left(\bigcup_{i=1}^{\infty}\bigcap_{k=i}^{\infty}\left\{x\in \Omega: |f_k(x)|>t\right\}\right)\bigcup \left(\Omega\setminus \mathcal{L}_{f}\right) .
	\]
Moreover, it follows from the increasing set property of $\widehat{\mathcal{H}}_{\infty}^{s(n-\delta)}$ (cf. \cite[Theorem 2.1]{YY08}) that 
	\[
	\widehat{\mathcal{H}}_{\infty}^{s(n-\delta)}\left(\bigcup_{i=1}^{\infty}\bigcap_{k=i}^{\infty}\left\{x\in \Omega:|f_k(x)|>t\right\}\right) =\lim_{i\rightarrow \infty}\widehat{\mathcal{H}}_{\infty}^{s(n-\delta)}\left(\bigcap_{k=i}^{\infty}\left\{x\in \Omega: |f_k(x)|>t\right\}\right).
	\]
Since we obtain from Theorem \ref{thm:fine property} that $\widehat{\mathcal{H}}_{\infty}^{s(n-\delta)}(\Omega\setminus\mathcal L_f)=0$,  the monotonicity of $\widehat{\mathcal{H}}_{\infty}^{s(n-\delta)}$ gives that for every $t>0$,
	\[
	\widehat{\mathcal{H}}_{\infty}^{s(n-\delta)}\left(\left\{x\in \Omega: |f^*(x)|>t\right\}\right)\leq \liminf_{k\rightarrow \infty}\widehat{\mathcal{H}}_{\infty}^{s(n-\delta)}\left(\left\{x\in \Omega: |f_k(x)|>t\right\}\right).
	\]
	By Fatou's lemma, we get
	\[
	\int_{0}^{\infty}\widehat{\mathcal{H}}_{\infty}^{s(n-\delta)}\left(\left\{x\in \Omega: |f^*(x)|>t\right\}\right)\,dt\leq \liminf_{k\rightarrow \infty}\int_{0}^{\infty}\widehat{\mathcal{H}}_{\infty}^{s(n-\delta)}\left(\left\{x\in \Omega: |f_k(x)|>t\right\}\right)\,dt.
	\]
This implies that
	\[
		\int_{\Omega} |f^*| \,d\widehat{\mathcal{H}}_{\infty}^{s(n-\delta)}\leq  \liminf_{k\to\infty} 	\int_{\Omega} |f_k| \,d\widehat{\mathcal{H}}_{\infty}^{s(n-\delta)}.
	\]
	As $\{f_k\}_{k=1}^{\infty}$ converge to $f$ in $\mathcal {W}^{\delta,1}(\Omega)$,  we obtain from \eqref{smooth} that
	\[
	\int_{\Omega} |f^*| \,d\widehat{\mathcal{H}}_{\infty}^{s(n-\delta)}
	\leq C_6(1-\delta) \liminf_{k\to\infty} \, [f_k]_{\mathcal {W}^{\delta,1}(\Omega)}
	\leq C_6(1-\delta) [f]_{\mathcal {W}^{\delta,1}(\Omega)},
	\] 
	which gives \eqref{eq:f^*} as desired by taking $C=C_6$ and finishes the proof.
\end{proof}

With the help of the functional Boxing inequality, we now prove Theorem \ref{geometric-boxing}.
\begin{proof}[Proof of Theorem \ref{geometric-boxing}]
	Let $U\subset \Omega$ be Lebesgue measurable with $|U|/|\Omega|\le \gamma$. If  $\chi_{U}\notin \mathcal {W}^{\delta,1}(\Omega),$ then the boxing inequality holds automatically because $\mathcal P_{\delta, \tau}(U, \Omega)=\infty$. Hence we may assume that $\chi_{U}\in \mathcal {W}^{\delta,1}(\Omega)$.  Let $\chi_{U}^{\ast }$ be the precise representative of $\chi_{U}$ and  $\mathcal{N} _{U}:=\left\{x\in\Omega:\chi_U\neq \chi_{U}^{\ast }\right\}.$ Then it follows from the Lebesgue differentiation theorem (cf. \cite[Theorem 1.33]{EGbook}) that $|\mathcal{N} _{U}|=0$.
	Moreover, it is clear that
	\[
	\vint{\Omega}\chi_{U}\,dx=\vint{\Omega}\chi_{U}^{\ast }\, dx\quad \text{and} \quad
	[\chi_{U}^{\ast }]_{\mathcal {W}^{\delta,1}(\Omega)}=[\chi_{U}]_{\mathcal {W}^{\delta,1}(\Omega)}=\mathcal P_{\delta, \tau}(U, \Omega).
	\]
	Therefore,
	\begin{align}\label{eq:chi-U}
		\left(1-\vint{\Omega}\chi_{U}\, dx\right) \mathcal{H}_{\infty} ^{s(n-\delta)}(U\setminus \mathcal{N} _{U})
		&=\int_{U\setminus \mathcal{N} _{U}}\left(\chi_{U}-\vint{\Omega}\chi_{U}\, dx\right) \,d\mathcal{H}_{\infty} ^{s(n-\delta)}\notag\\
		&= \int_{U\setminus \mathcal{N} _{U}}\left(\chi^\ast_{U}-\vint{\Omega}\chi^\ast_{U}\, dx\right) \,d\mathcal{H}_{\infty} ^{s(n-\delta)}.
	\end{align}
	By taking $f=\chi_{U}$ in Theorem \ref{functional-boxing}, we obtain from $U\setminus \mathcal{N} _{U}\subset \Omega$ that
	\begin{align*}
		\left(1-\vint{\Omega}\chi_{U}\, dx\right) \mathcal{H}_{\infty} ^{s(n-\delta)}(U\setminus \mathcal{N} _{U})
		&\leq \int_{\Omega}\left(\chi_{U}^{\ast }-\vint{\Omega}\chi_{U}^{\ast}\, dx\right) \,d\mathcal{H}_{\infty} ^{s(n-\delta)}\\
		&\leq C_{F}(1-\delta)[\chi_{U}^{\ast }]_{\mathcal {W}^{\delta,1}(\Omega)}=C_{F}(1-\delta)\mathcal P_{\delta, \tau}(U,\Omega),
	\end{align*}
	where $C_F=C_F(n, \tau, s, C_J, \diam\Omega)$ is the positive constant in Theorem \ref{functional-boxing}.
	Since  the assumption $|U|/|\Omega|<\gamma$ implies that $1-\vint{\Omega}\chi_{U}\, dx>1-\gamma$,  we have
	\[
	\mathcal{H}_{\infty} ^{s(n-\delta)}(U\setminus \mathcal{N} _{U})\leq \frac{C_F}{\left(1-\vint{\Omega}\chi_{U}\right)}(1-\delta)\mathcal P_{\delta, \tau}(U,\Omega)\leq \frac{C_F}{1-\gamma}(1-\delta)\mathcal P_{\delta, \tau}(U,\Omega).
	\]
	Therefore, taking $C=C_F/(1-\gamma)$ completes the proof of the first part of the theorem. 
	
	If $U$ is an open set, then $\chi_{U}(x)=\chi_{U}^{\ast }(x)$ for every $x\in U$. Hence,
	$$	\left(1-\vint{\Omega}\chi_{U}\, dx\right) \mathcal{H}_{\infty} ^{s(n-\delta)}(U)=\int_U \left(\chi^\ast_{U}-\vint{\Omega}\chi^\ast_{U}\, dx\right) \,d\mathcal{H}_{\infty} ^{s(n-\delta)}.$$ 
	Substituting the estimate above into \eqref{eq:chi-U} and following the remainder of the argument yields the conclusion.
\end{proof}
In the course of the proof of Theorem \ref{geometric-boxing}, we actually find that the functional Boxing inequality implies the geometric Boxing inequality and vice versa. That is, the functional and geometric Boxing inequalities are equivalent. The precise statement is as follows.
\begin{prop}\label{prop-equivalence}
{\it Let $\delta, \tau\in (0, 1)$ and $n-\delta\leq \beta\leq n$.  Then the following two conditions are equivalent.
\begin{itemize}
	\item[(1)] For any $\gamma\in (0, 1)$, there exists a constant $C>0$ independent of $\delta$ such that for every Lebesgue measurable set  $U\subset \Omega$ with $|U|/|\Omega|\leq \gamma$, one can find a Lebesgue measure zero set $\mathcal{N}_{U}$ such that 
		\begin{align*}
			\mathcal{H}_{\infty} ^{\beta}(U\setminus \mathcal{N}_{U})\le C(1-\delta)\mathcal P_{\delta, \tau}(U,\Omega).
	\end{align*} 
					If $U$ is an open set, then
							\begin{align*}
						\mathcal{H}_{\infty} ^{\beta}(U)\le C(1-\delta)\mathcal P_{\delta, \tau}(U,\Omega).
					\end{align*} 
	\item[(2)] There exists a constant $C>0$ independent of $\delta$ such that for every $f\in  L^1(\Omega)$, 
	\begin{align*}
			\int_{\Omega}\left\lvert f^*-\vint{\Omega} f\right\rvert \,d\mathcal{H}_{\infty} ^{\beta} \leq C(1-\delta)\int_\Omega \int_{|x-y|<\tau d(y)}\frac{\left\lvert f(x)-f(y)\right\rvert }{\left\lvert x-y\right\rvert^{n+\delta} }\,dx\,dy.
		\end{align*}
\end{itemize} }
\end{prop}
	\begin{proof}
		(1) $\Rightarrow$ (2). Fix a point $x_0\in \Omega$ and let $B_0=B(x_0,d(x_0)/18)$. For any open set $U\subset\Omega$ with $U\cap B_0=\emptyset,$ we have
		$$\frac{|U|}{|\Omega|}\leq \frac{|\Omega\setminus B_0|}{|\Omega|}=1-\frac{|B_0|}{|\Omega|}=:\gamma_0.$$
		Therefore, applying (1) to $U$ with $\gamma=\gamma_0$,
		we obtain an estimate similar to Proposition \ref{prop:sJ-box}. 
		
		Then we incorporate this estimate into the proof of Proposition \ref{lem:sJ-WPI-B0}, replace the Hausdorff content $\mathcal{H}_{\infty}^{n-\alpha}$ by $\mathcal{H}_{\infty}^{\beta}$, and repeat the same argument as in the proof of Proposition \ref{lem:sJ-WPI-B0}. Consequently, we obtain that
        there exists a constant $C>0$ independent of $\delta$ such that the following inequality 
		\[
		\int_{\Omega}\left\lvert f\right\rvert \,d\mathcal{H}_{\infty} ^{\beta} \leq C(1-\delta)\int_\Omega \int_{|x-y|<\tau d(y)}\frac{\left\lvert f(x)-f(y)\right\rvert }{\left\lvert x-y\right\rvert^{n+\delta} }\,dx\,dy,
		\]
		holds for every $f\in C(\Omega)\cap L^1(\Omega)$ with $f|_{B_0}=0$.
		
		Moreover, applying the same argument as in the proof of Theorem \ref{thm:sJ-WPI} to the above inequality, we obtain that there exists a constant $C^{\prime}>0$ independent of $\delta$ such that the following inequality 
		\[
		\int_{\Omega}\left\lvert f-\vint{3B_0}f\right\rvert \,d\mathcal{H}_{\infty} ^{\beta} \leq C^{\prime}(1-\delta)\int_\Omega \int_{|x-y|<\tau d(y)}\frac{\left\lvert f(x)-f(y)\right\rvert }{\left\lvert x-y\right\rvert^{n+\delta} }\,dx\,dy,
		\]
        holds for every $f\in C(\Omega)\cap L^1(\Omega)$. Repeating the same argument as in Lemmas~\ref{lem:1J-SPI} and \ref{lem:1J-Fbox} yields  a constant $C^{\prime\prime}>0$  independent of $\delta$ such that for every
        {$f\in L^1(\Omega)\cap C(\Omega)$,}
        \[
		\int_{\Omega}\left\lvert f-\vint{\Omega}f\right\rvert \,d\mathcal{H}_{\infty} ^{\beta} \leq C^{\prime\prime}(1-\delta)\int_\Omega \int_{|x-y|<\tau d(y)}\frac{\left\lvert f(x)-f(y)\right\rvert }{\left\lvert x-y\right\rvert^{n+\delta} }\,dx\,dy.
		\]
       Since $\beta\geq n-\delta$, by using Theorem~\ref{thm:fine property} and the same density argument as in the proof of Theorem \ref{functional-boxing},  we complete the proof of (1) $\Rightarrow$ (2).

		(2) $\Rightarrow$ (1). The proof of this direction is the same as the proof of Theorem \ref{geometric-boxing}.

{Furthermore, the arguments show that in both directions, the constants appearing in the inequalities are independent of $\delta$. This completes the proof.}
	\end{proof}

\section{\large Poincar\'e--Wirtinger trace  inequalities and John characterizations}\label{Chapter-6}
In this section, we give the proof of the Poincaré--Wirtinger trace inequality (Corollary~\ref{intro:s-john-gFPI}), which is one of the most important results derived from the functional Boxing inequality (Theorem~\ref{functional-boxing}).  We then prove Theorem~\ref{Box-John}.




We first introduce a necessary definition.  The space of {\it functions of bounded variation} $BV(\Omega)$ is defined as
\[
BV(\Omega):=\left\{f\in L^{1}(\Omega):\left\lVert Df\right\rVert (\Omega):=\sup\left\{ \int_\Omega f \, \mathrm{div}\phi \, dx \mid \phi \in C_c^1(\Omega; \mathbb{R}^n), |\phi| \leq 1 \right\} < \infty \right\} .
\]
If \(f\in BV(\Omega)\), then $f$ is called a function of bounded variation (abbreviated as BV function). For relevant properties of BV functions, we refer the interested reader to the monograph \cite{EGbook}.

We are ready to prove Corollary \ref{intro:s-john-gFPI}.  From the proof, one may capture  the critical role of the BBM factor \(1-\delta\).

\begin{proof}[Proof of Corollary \ref{intro:s-john-gFPI}]
It suffices to show	that \eqref{introeq:smooth-1} holds for every $f\in BV(\Omega)$. We first prove this for $f\in C^{\infty}(\Omega)\cap BV(\Omega).$
To this end,  we split the proof into the following two cases.

{\bf Case 1:} $s=\frac{n}{n-1}$. In this case, by repeating the proof of \cite[Corollary 5]{HK98}, we obtain that there exists a constant $C_{8}^\prime=C_{8}^\prime(n, C_J, \diam \Omega)$ such that 
$$	\left\lVert f-\vint{\Omega}f\right\rVert _{L^{1}(\Omega)} \leq C_{8}^\prime\int_{\Omega} |Df|\, dx.$$
Since \eqref{eq:growth-estimate-intro-2} implies that $\mu(A)\leq C(C_\mu, n) |A|$ for every Borel set $A\subset \Omega$. Therefore, 
$$	\left\lVert f-\vint{\Omega}f\right\rVert _{L^{q}(\Omega,d\mu)}\leq C(C_\mu, n) \left\lVert f-\vint{\Omega}f\right\rVert _{L^{1}(\Omega)}, $$
which gives the inequality \eqref{introeq:smooth-1} with $C=C(C_\mu, n) C_{8}^\prime$. 

{\bf Case 2:} $1\leq s<\frac{n}{n-1}$. In this case, $1\leq s\leq \frac{n}{n-\delta}$ whenever $\delta$ is sufficiently close to $1$. Therefore, we are able to use the fractional Poincar\'e--Wirtinger trace inequality in Corollary \eqref{cor:1-john-Trace-intro} and the Bourgain--Brezis--Mironescu formula \eqref{intoBBM-ar} on an arbitrary domain  to derive \eqref{introeq:smooth-1}. For each $k\in \mathbb N$, let  $\mathcal{D}_k$ be the collection of dyadic cubes of $\IR^n$ with the side length $2^{-k}.$ 
Since $s(n-1)\in [n-1, n)$, there exists a constant $C_\mu^\prime=C(C_\mu, n)$ such that for each $k\in \mathbb Z$ and each dyadic cube $P\in \mathcal D_k$, 
\begin{equation}\label{mu-cube}
	\mu(P)\leq C_\mu^\prime \ell(P)^{s(n-1)},
\end{equation}
where $\ell(P)$ denotes the side length of the dyadic cube $P$. 

We define a truncation measure $\mu_k$ of the measure \(\mu\) by setting
\[
\mu_k(U):=\sum_{P\in \mathcal{D}_k}\frac{|U\cap P|}{|P|}\mu(P).
\]
For every $1<p<\frac{n}{n-1}$, let $\alpha_p=ps(n-1)$. Then $0<\alpha_p<s(n-1)$ and $\alpha_p\rightarrow s(n-1)$ as $p\rightarrow 1^+$. For each dyadic cube $P'\in \mathcal{D}_{k'}$, if $k'\leq k$, it follows from \eqref{mu-cube} that
\begin{align}\label{es-mu_k-1}
	\mu_k(P')&=\sum_{P\in \mathcal{D}_k}\frac{|P' \cap P|}{|P|}\mu(P')=\mu(P')\leq C^\prime_\mu \ell(P')^{s(n-1)}\notag\\
	&=C^\prime_{\mu}2^{k'(s(n-1)-\alpha_p)}\ell(P')^{\alpha_p}\leq C^\prime_{\mu}2^{k(s(n-1)-\alpha_p)}\ell(P')^{\alpha_p}.
\end{align}
If $k'>k$, there exist a  unique $P_k\in \mathcal{D}_k$ such that $P'\subset P_k$. Then it follows from \eqref{mu-cube} that
\begin{align}\label{es-mu_k-2}
	\mu_k(P')&=\sum_{P\in \mathcal{D}_k}\frac{|P' \cap P|}{|P|}\mu(P')=\frac{|P'|}{|P_k|}\mu(P_k)=2^{kn-k'n}\mu(P_k)\leq C'_\mu 2^{kn-k'n}\ell(P_k)^{s(n-1)}\notag\\
	&=C'_\mu2^{ks(n-1)-k'\alpha_p}\ell(P')^{\alpha_p}\leq C'_\mu 2^{k(s(n-1)-\alpha_p)} 2^{\alpha_p(k-k')} \ell(P')^{\alpha_p}\notag\\
	&\leq C'_\mu2^{k(s(n-1)-\alpha_p)} \ell(P')^{\alpha_p}.
\end{align}
By  \eqref{es-mu_k-1} and \eqref{es-mu_k-2}, we obtain that for each dyadic cube $P'$, 
$$\mu_k(P')\leq C'_\mu2^{k(s(n-1)-\alpha_p)} \ell(P')^{\alpha_p}.$$
Then by a standard covering argument,  we have $\mu_k\leq 2^{k(s(n-1)-\alpha_p)}C'_\mu\widehat {\mathcal{H}}_\infty^{\alpha_p}$.  Since we know from \eqref{dyadic-comparable} that $\mathcal H^{\alpha_p}_\infty$ is equivalent to $\widehat {\mathcal{H}}_\infty^{ \alpha_p}$ with equivalence constants depending only on $n$, there exists a constant $C_\mu^{''}=C(C_\mu, n)$ 
{such that $\mu_k\leq C_\mu^{''} 2^{k(s(n-1)-\alpha_p)} \mathcal{H}^{\alpha_p}_\infty$.} 
For each fixed $p>1$ sufficiently close to $1$, let
\[
p_\delta=\frac{p(n-1)}{n-\delta}.
\]
Since $p>1$ and $s<\frac{n}{n-1}$, we have
\[
1\leq p_\delta\leq \frac{n}{s(n-\delta)}
\]
whenever $\delta$ is sufficiently close to $1$.
Moreover, since
\[
p_\delta s(n-\delta)=ps(n-1)=\alpha_p,
\]
we may apply \eqref{cor:1-john-Trace-intro-q}
to obtain
\begin{align}\label{eq:mu-k-fractional}
\left\|f-\vint{\Omega} f\right\|_
{L^{p_\delta}(\Omega,d\mu_k)}
&\leq
\left(C_\mu''2^{k(s(n-1)-\alpha_p)}\right)^{1/p_\delta}
C_{PW}(1-\delta)
\int_\Omega
\int_{|x-y|<\tau d(y)}
\frac{|f(x)-f(y)|}
{|x-y|^{n+\delta}}\,dx\,dy,
\end{align}
where $C_{PW}$ is independent of $\delta$
and $p$ sufficiently close to $1$.

We now let $\delta\to1^-$.
By the Bourgain--Brezis--Mironescu formula
\eqref{intoBBM-ar}, we have
\[
\lim_{\delta\to1^-}
(1-\delta)
\int_\Omega
\int_{|x-y|<\tau d(y)}
\frac{|f(x)-f(y)|}
{|x-y|^{n+\delta}}\,dx\,dy
=
K_n\int_\Omega |Df|\,dx.
\]
Since $p_\delta\to p$ as $\delta\to1^-$,
Fatou's lemma and \eqref{eq:mu-k-fractional}
yield
\begin{align*}
\int_\Omega
\left|f-\vint{\Omega} f\right|^p\,d\mu_k
&\leq
\liminf_{\delta\to1^-}
\int_\Omega
\left|f-\vint{\Omega} f\right|^{p_\delta}
\,d\mu_k\\
&\leq
C_\mu''2^{k(s(n-1)-\alpha_p)}
(C_{PW}K_n)^p
\left(\int_\Omega |Df|\,dx\right)^p.
\end{align*}
Consequently,
\begin{align}\label{eq:mu-k-p}
\left\|f-\vint{\Omega} f\right\|_{L^p(\Omega,d\mu_k)}
\leq
\left(C_\mu''2^{k(s(n-1)-\alpha_p)}\right)^{1/p}
C_{PW}K_n\int_\Omega |Df|\,dx.
\end{align}

we obtain that  
\begin{align*}
	\left\| f - \vint{\Omega} f \right\|_{L^{p}(\Omega, d\mu_k)}  \leq 2^{k(s(n-1)-\alpha_p)}C_\mu^{''}C_{PW}K_n\int_{\Omega}|Df|\,dx.
\end{align*}

{By letting $p\to 1^{+}$, so that $\alpha_p\rightarrow s(n-1)$,}
\begin{align}\label{mu_k}
\left\| f - \vint{\Omega} f \right\|_{L^{1}(\Omega, d\mu_k)} \leq C_\mu^{''}C_{PW}K_n\int_{\Omega}|Df|\,dx.
\end{align}

For any open set $U$,
via the definition of $\mu_k$ it is clear that $\mu(U)\le \liminf_{k\to\infty}\mu_k(U)$.
Thus, by Cavalieri's principle and Fatou's lemma, it follows from \eqref{mu_k} that
\begin{align*}
\int_{\Omega}\left\lvert f-\vint{\Omega} f\right\rvert \,d\mu
&=\int_{0}^{\infty}\mu\left(\left\{x\in \Omega:|f(x)-f_{\Omega}|>t\right\} \right)\,dt\\
&\le \int_{0}^{\infty}\liminf_{k\to\infty}\mu_k\left(\left\{x\in \Omega:|f(x)-f_{\Omega}|>t\right\} \right)\,dt\\
&\leq \liminf_{k\to\infty}\int_{0}^{\infty}\mu_k\left(\left\{x\in \Omega:|f(x)-f_{\Omega}|>t\right\} \right)\,dt\\
&=\liminf_{k\to\infty}\int_{\Omega}\left\lvert f-\vint{\Omega} f\right\rvert \,d\mu_k
\leq C_\mu^{''}C_{PW}K_n\int_{\Omega}|Df|\,dx,
\end{align*}
which gives \eqref{introeq:smooth-1} with $C= C_\mu^{''}C_{PW}K_n$.

Therefore, we have proved that the conclusion of the corollary
holds for any $f\in  C^{\infty}(\Omega)\cap BV(\Omega).$ 
Now for any $f\in BV(\Omega)$, we know that $f$ can be approximated by smooth functions (cf. \cite[Theorem 5.3]{EGbook}), i.e. there exist functions $\{f_k\}_{k=1}^\infty \subset BV(\Omega) \cap C^\infty(\Omega)$ such that
\begin{align*}
f_k \to f~ \text{in}~ L^1(\Omega)\quad\text{and}\quad\|D f_k\|(\Omega) \to \|D f\|(\Omega) ~\text{as}~ k \to \infty.
\end{align*}
Moreover, from the proof of \cite[Theorem 5.3]{EGbook}, the smooth functions $f_k$ can be constructed by using the standard mollification as we did in Proposition \ref{prop:fracSob-dense}.
Thus, according to \cite[Theorem 5.20]{EGbook}, the same reasoning as the one in Remark \ref{rem-dense}  implies that for $\mathcal{H}^{n-1}$-a.e. $x\in\Omega$
\[
f^{*}(x)=\lim_{k\to \infty}f_k(x).
\]
Since $n-1<s(n-1)\leq n$, then every $\mathcal{H}^{n-1}$ measure zero set is $\mu$-measure zero,
 repeat  the density arguments  in Theorem \ref{functional-boxing} by replacing the Hausdorff content to Radon measure $\mu$  and we obtain that \eqref{introeq:smooth-1} holds for every $f\in BV(\Omega)$. Therefore, the proof is complete.
\end{proof}

    Before proving Theorem \ref{Box-John}, we first state a crucial result, which follows from Guo \cite[Theorem 1.4]{Guo17} by taking $p=1$. Let $\delta, \tau\in (0, 1)$. A domain  $\Omega\subset \mathbb R^n$ is called a fractional $(q, 1)$-Sobolev-Poincar\'e domain, if there exists a constant $C>0$ such that 
    \begin{align}\label{q-FSP}
    	\left\| f - \vint\Omega f \right\|_{L^{q}(\Omega)}\leq C\int_\Omega \int_{\left\lvert x-y\right\rvert <\tau d(y)}\frac{\left\lvert f(x)-f(y)\right\rvert }{\left\lvert x-y\right\rvert^{n+\delta} }\,dx\,dy
    \end{align}
    holds for every $f\in C(\Omega)\cap L^1(\Omega)$. 
     
\begin{thm}\label{Guo-lem}
Let $\delta, \tau\in (0, 1)$. Assume that $\Omega \subset \mathbb{R}^n$ is a bounded domain that satisfies the separation property.   If $\Omega$ is a fractional $(q,  1)$-Sobolev-Poincar\'e domain for some $q > 1$, then $\Omega$ is an $s$-diameter John domain with $s=\delta/\big((n-\delta)(q-1)\big)$.
\end{thm}

We are ready to prove Theorem \ref{Box-John}.
\begin{proof}[Proof of Theorem \ref{Box-John}]
Given $n-\delta\leq \beta<n$, 
 repeating the arguments in the proof of (1) $\Rightarrow $ (2) in Proposition~\ref{prop-equivalence} yields the following functional Boxing inequality: there exists a constant $C>0$ such that for every $f\in C(\Omega)\cap L^1(\Omega)$, 
	\begin{align*}
			\int_{\Omega}\left\lvert f-\vint{\Omega} f\right\rvert \,d\mathcal{H}_{\infty} ^{\beta} \leq C(1-\delta)\int_\Omega \int_{\left\lvert x-y\right\rvert <\tau d(y)}\frac{\left\lvert f(x)-f(y)\right\rvert }{\left\lvert x-y\right\rvert^{n+\delta} }\,dx\,dy.
		\end{align*}
		
 By the same reasoning as in the proof of  
 Theorem \ref{cor:1-john-Trace-intro}
 with $\mu$  the Lebesgue measure and $q=n/\beta>1$, we obtain  that $\Omega$ is a fractional $(q, 1)$-Sobolev-Poincar\'e domain. Therefore, it follows from Theorem~\ref{Guo-lem} that $\Omega$ is an $s$-diameter John domain with $$s=\frac{\delta}{(n-\delta)(q-1)}=\frac{\delta\beta}{(n-\delta)(n-\beta)},$$ 
 which completes the proof.
\end{proof}

\section{\large Results on special domains: $1$-John domains and Lispchitz domains}\label{John and Lip}
In this section, we make a conclusion about the all obtained inequalities on special domains: $1$-John domains and Lipschitz domains.
\medskip

\noindent\textbf{Results on $1$-John domains.} By taking $s=1$, we obtain a collection of inequalities on $1$-John domains, which we state in the following. Since in this case these inequalities are scaling invariant, their constants are independent of $\diam \Omega$; see also Remark~\ref{rem:C_1}.

Let $\Omega\subset\mathbb R^n$ be a $1$-John domain with John constant $C_J\geq 1$. Let $\delta,\tau\in(0,1)$. Assume that $\mu$ is a Radon measure on $\mathbb R^n$ with
$\mu(B(x,r))\leq C_\mu r^{n-\delta}$ for every ball $B(x,r)$.
For any $f\in L^1(\Omega)$, recall that
$\vint{\Omega} f=\frac{1}{|\Omega|}\int_\Omega f(x)\,dx$
and $f^*$ is the precise representative of $f$. Then there exists a constant $C$ depending only on $n,\tau$ and $C_J$ (additionally on $\gamma$ in (1), and on $C_\mu$ in (4) and (7)) such that the following inequalities hold.
\begin{itemize}
    \item[(1)] {\it Geometric boxing inequality.} If $U\subset\Omega$ is a Lebesgue measurable set with $|U|/|\Omega|\leq\gamma\in(0,1)$, then there exists a Lebesgue measure zero set $\mathcal N_U$ such that
    \begin{align*}
        \mathcal H_\infty^{n-\delta}(U\setminus\mathcal N_U)
        \leq C(1-\delta)\mathcal P_{\delta,\tau}(U,\Omega).
    \end{align*}

    \item[(2)] {\it Functional boxing inequality.} For every $f\in L^1(\Omega)$,
    \begin{align*}
        \int_\Omega\left|f^*-\vint{\Omega} f\right|
        \,d\mathcal H_\infty^{n-\delta}
        \leq C(1-\delta)
        \int_\Omega\int_{|x-y|<\tau d(y)}
        \frac{|f(x)-f(y)|}{|x-y|^{n+\delta}}\,dx\,dy.
    \end{align*}


    \item[(3)] {\it Fractional Poincar\'e--Wirtinger inequality.} For every $f\in L^1(\Omega)$,
    \begin{align*}
        \left\|f^*-\vint{\Omega} f\right\|_{L^1(\Omega,d\mu)}
        \leq C(1-\delta)
        \int_\Omega\int_{|x-y|<\tau d(y)}
        \frac{|f(x)-f(y)|}{|x-y|^{n+\delta}}\,dx\,dy.
    \end{align*}

    \item[(4)] {\it Fractional Sobolev--Poincar\'e inequality.} For every $f\in L^1(\Omega)$,
    \begin{align*}
        \left\|f-\vint{\Omega} f\right\|_{L^{n/(n-\delta)}(\Omega)}
        \leq C(1-\delta)
        \int_\Omega\int_{|x-y|<\tau d(y)}
        \frac{|f(x)-f(y)|}{|x-y|^{n+\delta}}\,dx\,dy.
    \end{align*}

    \item[(5)] {\it Fractional Hardy-type inequality.} For every $f\in L^1(\Omega)$ with $\vint{\Omega} f=0$,
    \begin{align}
        \int_\Omega |f(x)|\frac{dx}{|x|^\delta}
        \leq C(1-\delta)
        \int_\Omega\int_{|x-y|<\tau d(y)}
        \frac{|f(x)-f(y)|}{|x-y|^{n+\delta}}\,dx\,dy.
    \end{align}

    \item[(6)] {\it Poincar\'e--Wirtinger inequality.}
    Here we require $\mu(B(x,r))\leq C_\mu r^{n-1}$ for every ball $B(x,r)$ instead. For every $f\in BV(\Omega)$,
    \begin{align*}
        \left\|f^*-\vint{\Omega} f\right\|_{L^1(\Omega,d\mu)}
        \leq C\|Df\|(\Omega).
    \end{align*}
\end{itemize}

For the converse, it follows from Corollary \ref{Box-relative}, Proposition \ref{prop-equivalence} and Theorem \ref{Guo-lem} that, under the separation property, $\Omega$ is a $1$-John domain whenever it supports any of the inequalities (1), (2),  or (4). Moreover, since (4) is a special case of (3), and (5) and (6) are derived from (3),
we conclude that the $1$-John domain condition is essentially the sharp condition for domains to
support the inequalities (1)--(6). 
\medskip

\noindent\textbf{Results on Lipschitz domains.} 
Since every Lipschitz domain is $1$-John (with the John constant $C_J\ge 1$ depending only on its geometry),  the inequalities (1)--(6) automatically hold. In particular, since $[f]_{\mathcal W^{\delta,1}(\Omega)}\le [f]_{W^{\delta,1}(\Omega)}$, we may replace $\mathcal P_{\delta,\tau}(U,\Omega)$ on the right-hand side of (1) by $P_{\delta}(U,\Omega)$, and replace the improved fractional Sobolev energy on the right-hand side of (2), (3), (4) and (5) by the classical fractional Sobolev energy. We omit the detailed statements here for brevity. Moreover, the constants $C$ in these inequalities depend only on $n$ and $C_J$ (additionally on $\gamma$ in (1) and on $C_\mu$ in (3) and (6)). To the best of our knowledge, even on  Lipschitz domains, the inequalities (1)--(5) with the BBM factor $1-\delta$  are new.

\vspace{1cm}

\noindent\textbf{Acknowledgments.} 
The authors are grateful to the referee for pointing out the duality argument that allows the general \(q\)-case to be deduced from the case \(q=1\).
\bigskip

\noindent\textbf{Funding.} Manzi Huang was partly supported by NSFC under the Grant No. 12371071 and the Key Project of NSF of Hunan Province under Grant No. 2026JJ30002. Jiang Li was partly supported by NSFC under Grant No. 12571081. Zhuang Wang was partly supported by NSF of Hunan Province under Grant No. 2024JJ6299, the Scientific Research Fund of Hunan Provincial Education Department  under Project No. 25B0095, and  NSFC under Grant No. 12101226.


\bigskip


\begin{thebibliography}{99}

\bibitem{Adams98}
D. R. Adams, {\it Choquet integrals in potential theory}, Publ. Mat. 42 (1998), 3--66.


\bibitem{BBM01}
J. Bourgain, H. Brezis, and P. Mironescu, 
\textit{ Another look at Sobolev spaces}, in: Optimal Control and Partial Differential Equations, IOS, Amsterdam, 2001, pp. 439-455.

\bibitem{BBM02}
J. Bourgain, H. Brezis, and P. Mironescu, 
\textit{Limiting embedding theorems for $W^{s,p}$ when $s\uparrow 1$ and applications},
J. Anal. Math., 87 (2002), 77--101.

\bibitem{BLP14}
L. Brasco, E. Lindgren, and E. Parini, {\it The fractional Cheeger problem},
Interfaces Free Bound. 16 (2014), no. 3, 419--458.


\bibitem{BP95}
S. Buckley, and P. Koskela, {\it Sobolev-Poincar\'e implies John}, Math. Res. Lett. 2 (1995),   577--593.

\bibitem{CaRoSa10}
L. Caffarelli, J. M. Roquejoffre, and O. Savin, \textit{Non-local minimal surfaces},
Comm. Pure Appl. Math. 63 (2010), 1111--1144.

\bibitem{Davila02}
J. Dávila, \textit{ On an open question about functions of bounded variation},
Calc. Var. Partial Differential Equations 15 (2002), no. 4, 519-527.

\bibitem{DD22}
I. Drelichman, and R. G. Durán, {\it The Bourgain-Brézis-Mironescu formula in arbitrary bounded domains},
Proc. Amer. Math. Soc. 150 (2022), no. 2, 701--708.

\bibitem{Dyda06}
B. Dyda, \textit{On comparability of integral forms}, J. Math. Anal. Appl. 318 (2006), 564--577.

\bibitem{DyIhVa16}
B. Dyda, L. Ihnatsyeva, and A. V. Vähäkangas,
\textit{On improved fractional Sobolev-Poincaré inequalities},
Ark. Mat. 54, no. 2 (2016): 437--454.

\bibitem{EGbook}
L. C. Evans, and R. F. Gariepy,
\emph{Measure theory and fine properties of functions},
Textb. Math., CRC Press, Boca Raton, FL, 2015, xiv+299 pp.

\bibitem{Guo17}
C. Guo,
\emph{ Fractional Sobolev-Poincaré inequalities in irregular domains
}, Chinese Ann. Math. Ser. B 38 (2017), no. 3, 839-856.

\bibitem{Guo-Koskela}
C. Guo, and P. Koskela, \emph{Generalized John disks}, Cent. Eur. J. Math., 12 (2014), 349--361.

\bibitem{Gustin60}
W. Gustin, \emph{  Boxing inequalities}, J. Math. Mech. 9 (1960), 229-239.

\bibitem{HK98}
P. Haj\l{}asz, and P. Koskela, {\it Isoperimetric inequalities and imbedding theorems in irregular domains}, J. London Math. Soc. (2) 58 (1998), no. 2, 425-450.

\bibitem{HH23}
P. Harjulehto, and R. Hurri-Syrjänen, {\it On Choquet integrals and Poincaré-Sobolev inequalities}, J. Funct. Anal. 284 (2023), no. 9, Paper No. 109862, 18 pp.

\bibitem{Heinonen} J.~Heinonen, {\it Lectures on analysis on metric spaces}, Universitext. Springer-Verlag, New York, 2001. x+140 pp.


\bibitem{HV13}
R. Hurri-Syrjänen, and A. V. Vähäkangas, {\it  On fractional Poincaré inequalities}, J. Anal. Math. 120 (2013), 85-104.


\bibitem{John61}
F. John, {\it  Rotation and strain,} Comm. Pure Appl. Math., 14, 1961, 391-413. 

%



\bibitem{Leo23}
G.~Leoni, {\it  A first course in fractional Sobolev spaces,} American Mathematical Society, Providence, RI, 2023, xv+586 pp.

\bibitem{MS79}
O. Martio, and J. Sarvas, {\it   Injectivity theorems in plane and space}, Ann. Acad. Sci. Fenn. Ser. A I Math., 4(2), 1979, 383-401. 

\bibitem{Mattila}
P. Mattila, \emph{Geometry of sets and measures in Euclidean spaces,} Cambridge Stud. Adv. Math. 44, Cambridge University Press, Cambridge, 1995.

\bibitem{MS02}
V. Maz'ya, and T. Shaposhnikova, {\it  On the Bourgain, Brezis, and Mironescu theorem concerning limiting embeddings of fractional Sobolev spaces}, J. Funct. Anal. 195 (2002), no. 2, 230-238.

\bibitem{Mazya03}
V. Maz'ya, {\it  Lectures on isoperimetric and isocapacitary inequalities in the theory of Sobolev spaces}, In: ``Heat Kernels and Analysis on Manifolds, Graphs, and Metric Spaces (Paris, 2002)", Contemp. Math., Vol. 338, Amer. Math. Soc., Providence, RI, 2003, 307-340.

\bibitem{MZ77}
N. G. Meyers, and W. P. Ziemer, {\it   Integral inequalities of Poincaré and Wirtinger type for BV functions},
Amer. J. Math. 99 (1977), no. 6, 1345-1360.

\bibitem{M24}
K. Mohanta, {\it Bourgain-Brezis-Mironescu formula for $W^{s,p}_{q}$ spaces in arbitrary domains},
Calc. Var. Partial Differential Equations 63 (2024), no. 2, Paper No. 31, 17 pp.

\bibitem{MPW24}
K. Myyryläinen, C. Pérez, and J. Weigt,
\textit{Weighted fractional Poincaré inequalities via isoperimetric inequalities},
Calc. Var. Partial Differential Equations, 63 (2024), no. 8, Paper No. 205, 32 pp.

\bibitem{NV91}
R. Näkki, and J. Väisälä, 
\textit{ John disks},
Exposition. Math., 1991, 9(1), 3-43.

\bibitem{B-P16}
A. Ponce, {\it Elliptic PDEs, measures and capacities}, European Mathematical Society (EMS), Zürich, 2016, x+453 pp.

\bibitem{PS20}
A. Ponce, and D. Spector,
\textit{A boxing inequality for the fractional perimeter},
Ann. Sc. Norm. Super. Pisa Cl. Sci., (5) 20 (2020), no. 1, 107--141.

\bibitem{Saksman17}
M. Pratì, and E. Saksman, \textit{A $T(1)$ theorem for fractional Sobolev spaces on domains}, J. Geom. Anal. 27 (2017), 2490--2538.

\bibitem{Serra24}
J. Serra, \textit{Nonlocal minimal surfaces: recent developments, applications, and future directions},
SeMA J. 81 (2024), 165--191.

\bibitem{SS90}
W. Smith, and D. A. Stegenga, \textit{H\"older and Poincar\'e domains},
Trans. Amer. Math. Soc., 319, 1990, 67-100.

\bibitem{Vis91}
A. Visintin, \textit{Generalized coarea formula and fractal sets}, Japan J. Indust. Appl. Math. 8 (1991), 175-201. 

\bibitem{YY08}
D. Yang, and W. Yuan, {\it A note on dyadic Hausdorff capacities},
Bull. Sci. Math. 132 (2008), no. 6, 500--509.

\bibitem{Zhou15}
Y. Zhou, {\it Fractional Sobolev extension and imbedding},
Trans. Amer. Math. Soc. 367 (2015), 959--979.

\end{thebibliography}
\end{document}